\documentclass{amsart}
\usepackage{amsmath}
\usepackage{amssymb}
\usepackage{array}
\usepackage{amscd}
\usepackage{tikz-cd}
\usepackage{mathtools}

\numberwithin{equation}{section}
\newtheorem{Th}[subsection]{Theorem}
\newtheorem*{Th*}{Theorem}
\newtheorem{Lemma}[subsection]{Lemma}
\newtheorem{Prop}[subsection]{Proposition}

\newtheorem{Cor}[subsection]{Corollary}
\theoremstyle{definition}
\newtheorem{definition}[subsection]{Definition}
\newtheorem*{definition*}{Definition}
\newtheorem{Remark}[subsection]{Remark}
\newtheorem{Example}[subsection]{Example}
\newcommand{\comm}[1]{}

\usepackage{color}
\usepackage[normalem]{ulem}
\definecolor{DarkGreen}{rgb}{0,0.5,0.1}

\newcommand\soutD{\bgroup\markoverwith
{\textcolor{DarkGreen}{\rule[.5ex]{2pt}{1pt}}}\ULon}

\DeclarePairedDelimiter\floor{\lfloor}{\rfloor}

\makeatletter

\newcommand*{\rom}[1]{\expandafter\@slowromancap\romannumeral #1@}
\makeatother

\makeatletter
\newcommand{\doublewidetilde}[1]{{%
  \mathpalette\double@widetilde{#1}%
}}
\newcommand{\double@widetilde}[2]{%
  \sbox\z@{$\m@th#1\widetilde{#2}$}%
  \ht\z@=.9\ht\z@
  \widetilde{\box\z@}%
}
\makeatother

\begin{document}
\title[An elementary description of nef cone for IHSM]{An elementary description of nef cone for irreducible holomorphic symplectic manifolds}
\author{Anastasia V.~Vikulova}
\address{{\sloppy
\parbox{0.9\textwidth}{
Steklov Mathematical Institute of Russian
Academy of Sciences,
8 Gubkin str., Moscow, 119991, Russia
\\[5pt]
Laboratory of Algebraic Geometry, National Research University Higher
School of Economics, 6 Usacheva str., Moscow, 119048, Russia.
}\bigskip}}
\email{vikulovaav@gmail.com}
\date{}
\maketitle

	\begin{abstract}
	We describe  MBM classes  for irreducible holomorphic symplectic manifolds of $\mathrm{K3}$ and Kummer types. These classes are the monodromy images of extremal rational curves which give the faces of the nef cone of some birational model. We study the connection between our results and A.~Bayer and E.~Macr\`{\i}'s theory. We apply the numerical method of description  due to E.~Amerik and M.~Verbitsky in low dimensions to the~$\mathrm{K3}$ type and Kummer type cases.
	\end{abstract}

\maketitle

\tableofcontents
	
	\section{Introduction}
	
There is a great interest in studying ample cones in algebraic geometry in the context of minimal model program. Much had been done to find ample or K\"{a}hler cone of irreducible holomorphic symplectic manifolds. The simplest example of an  irreducible holomorphic symplectic manifolds is a $\mathrm{K3}$ surface. The self-intersection of smooth rational irreducible curves on $\mathrm{K3}$ surfaces is $-2.$ The K\"{a}hler cone of a $\mathrm{K3}$ surface consists of the classes of type $(1,1)$ in the second cohomologies with real coefficients,   which have positive self-intersection and which positively intersect with~$(-2)$-curves. The ample cone of a projective $\mathrm{K3}$ surface is characterized in the same way  in the N\'{e}ron-Severi group with real coefficients. So orthogonal hyperplanes to the classes of square $-2$ in the positive cone give the chamber decomposition of the positive cone, where the   K\"{a}hler cone (or the ample cone) is one of the chambers  (see, for example, the paper ~\cite[Corollary~3.4]{Huybrechts}).

However, it is not so easy to describe the ample or K\"{a}hler cone for other irreducible holomorphic symplectic manifolds. The recent results of A.~Bayer and E.~Macr\`{\i} use the difficult technique of the theory of  Bridgeland stability conditions.  There is a connection between wall-crossing for stability conditions and walls of the positive cone of a moduli space of $H$-Gieseker stable sheaves thanks to a natural morphism from the set of chambers to the N\'{e}ron-Severi group, constructed by A.~Bayer and E.~Macr\`{i} in the paper~\cite{BM12}. This technique was applied to all known examples of holomorphic symplectic manifolds (see \cite{BM},\cite{Bridgeland},\cite{MZ},\cite{Yoshioka}).  A.~Bayer and E.~Macr\`{\i} described the chamber decomposition of the positive cone of the moduli space of semistable sheaves on a~$\mathrm{K3}$ surface with primitive Mukai vector. In fact, such a moduli space is deformation equivalent to the Hilbert scheme of points on a~$\mathrm{K3}$ surface.  We mention two theorems from the papers ~\cite{BM} and ~\cite{Yoshioka}, which describe the ample cones of $\mathrm{K3}$ and Kummer types manifolds. 

Before we formulate the theorems we first introduce the following notations:
\begin{definition}
Let $X$ be a smooth projective surface. Denote by $H^*_{\text{alg}}(X,\mathbb{Z})$ the algebraic part of the ring of even cohomologies, i.e. 
$$
H^*_{\text{alg}}(X,\mathbb{Z})\,=\,H^0(X,\mathbb{Z}) \oplus \mathrm{NS}(X) \oplus H^4(X,\mathbb{Z}).
$$
\noindent For an element $v \in H^*_{\text{alg}}(X,\mathbb{Z})$ denote by $v^{\perp}$  the subspace of elements in $H^*_{\text{alg}}(S,\mathbb{Z})$ which are orthogonal to $v$ with respect to the Mukai pairing given by
$$
((r_1,\alpha_1,s_1),(r_2,\alpha_2,s_2))\,=\,\alpha_1 \cdot \alpha_2 - r_1 \cdot s_2 - r_2\cdot s_1,
$$

\noindent where $(r_i,\alpha_i,s_i) \in H^*_{\text{alg}}(X,\mathbb{Z})$ with $r_i \in H^0(X, \mathbb{Z}),$ $\alpha_i \in \mathrm{NS}(X)$ and~\mbox{$s_i \in H^4(X, \mathbb{Z}).$}

\end{definition}

\begin{Th}[{\cite[Theorem 12.1]{BM}}]\label{Theorem12.1} Let $S$ be a $\mathrm{K3}$ surface, $v \in H^*_{\text{alg}}(X,\mathbb{Z})$ be a fixed primitive Mukai vector and $M_{\sigma}(v)$ be the moduli space of $\sigma$-semistable sheaves on~$S$ with the Mukai vector $v$ and the stability condition $\sigma \in \mathrm{Stab}(S)$. Consider the set~$W$ of elements $a \in H^*_{\text{alg}}(S,\mathbb{Z})$ such that $a^2 \geqslant -2$ and~\mbox{$0 \leqslant (v,a) \leqslant \frac{v^2}{2}$} with respect to the Mukai pairing, where the element $a$ is identified with its image under the isometry 
$$
\theta:v^{\perp} \xrightarrow{\sim} \mathrm{NS}(M_{\sigma}(v))
$$
\noindent which is defined in~\cite[(1.6)]{YoshiokaAb}.  Then the set of $a^{\perp}$ with $a \in W$ defines a chamber decomposition of the positive cone of $M_{\sigma}(v).$  The ample cone $\mathrm{Amp}(M_{\sigma}(v))$ is one of the chambers: $\eta \in \mathrm{Amp}(M_{\sigma}(v))$ if and only if $(\eta,a)>0$ for all~\mbox{$a \in W.$}

\end{Th}
\begin{Th}[\cite{Yoshioka}]\label{Theorem12.1Kummer} Let $A$ be an abelian surface, $v \in H^*_{\text{alg}}(A,\mathbb{Z})$ be a fixed primitive Mukai vector, $M_{\sigma}(v)$ be the moduli space of $\sigma$-semistable sheaves on $A$ with the Mukai vector~$v$ and $\mathrm{Kum}_{\sigma}(v)$ be the fibre of the Albanese map 
\begin{align*}
M_{\sigma}(v) & \longrightarrow A \times A^{\vee}\\
\mathcal{F} &\longmapsto \left(\mathrm{Alb}(c_2(\mathcal{F})),\det(\mathcal{F})\right)
\end{align*}
\noindent over zero with $A^{\vee}\,=\,\mathrm{Pic}^0(A),$ where $\mathrm{Alb}:\mathrm{CH}_0(A) \to A$ is a summation map. Consider the set $W$ of elements $a \in H^*_{\text{alg}}(A,\mathbb{Z})$ such that~\mbox{$a^2 \geqslant 0$} and~\mbox{$0 < (v,a) \leqslant \frac{v^2}{2}$} with respect to the Mukai pairing, where  the element  $a$ is identified with its image under the isometry 
$$
\theta:v^{\perp} \xrightarrow{\sim} \mathrm{NS}(\mathrm{Kum}_{\sigma}(v))
$$
\noindent which is defined in~\cite[(1.6)]{YoshiokaAb}. Then the set of $a^{\perp}$ with $a \in W$ defines a chamber decomposition of the  positive cone of  $\mathrm{Kum}_{\sigma}(v).$ The ample cone~\mbox{$\mathrm{Amp}(\mathrm{Kum}_{\sigma}(v))$} is one of the chambers: $\eta \in \mathrm{Amp}(\mathrm{Kum}_{\sigma}(v))$ if and only if $(\eta,a)>0$ for all~\mbox{$a \in W.$}

\end{Th}

\begin{Remark}[{cf.~\cite{Beauville, Knutsen}}]
If $S$ is a $\mathrm{K3}$ surface, then  the Mukai vector $v=(1,0,0)$ defines the moduli space $M_{\sigma}(v) \simeq S$ for any $\sigma.$  If $A$ is a complex two-dimensional torus, then  the Mukai vector $v=(1,0,-2)$ defines the Kummer surface $\mathrm{Kum}_{\sigma}(v)$  isomorphic to  the blowup of the quotient of~$A$ by the natural involution for any $\sigma$.  

For a $\mathrm{K3}$ surface $S$ and Mukai vector $v=(1,0,-n+1)$ the moduli space $M_{\sigma}(v)$ is isomorphic to the Hilbert scheme of $n$ points $\mathrm{Hilb}^n(S)$ for any $\sigma.$ For a complex two-dimensional torus~$A$ and Mukai vector $v=(1,0,-n-1)$ the variety~$\mathrm{Kum}_{\sigma}(v)$ is isomorphic to the Kummer variety~\mbox{$\mathrm{Kum}^n(A)$} for any $\sigma.$ Recall that the Kummer variety for the two-dimensional torus $A$ is the fibre over $0 \in A$ of the summation map $\mathrm{Hilb}^{n+1}(A) \to A.$

\end{Remark}

All these theorems work for projective surfaces. In the non-algebraic case the ample cone is empty. However, the K\"{a}hler cone does always exist for irreducible holomorphic symplectic manifolds  and we have $Amp(X)\,=\,\mathcal{K}(X) \cap \mathrm{NS}(X)_{\mathbb{R}},$ where~$Amp(X)$ is the ample cone, $\mathcal{K}(X)$ is the K\"{a}hler cone and $\mathrm{NS}(X)$ is the N\'{e}ron-Severi group. The K\"{a}hler cone of irreducible holomorphic symplectic manifolds was studied by E.~Amerik, D.~Huybrechts, E.~Markman, M.~Verbitsky, and others (see, for instance,~\cite{AV, GJH, Markman}).

The structure of the ample cone for the Hilbert scheme of points on $\mathrm{K3}$ surface was originally conjectured in the papers of B.~Hassett and Y.~Tschinkel (for instance, see~\cite{Tschinkel}). A refinement of their conjecture is established by A.~Bayer and E.~Macr\`{\i}. In the paper of E.~Amerik and M.~Verbitsky~\cite{AV} there is more elementary description of the boundary of the K\"{a}hler cone for $\mathrm{K3}$ type manifolds  (i.e. deformation of the  Hilbert scheme of points on a $\mathrm{K3}$ surface) for small dimensions. This means that they describe MBM classes, i.e. the classes which correspond to the extremal rays of the Mori cone on some deformation of the variety (see Section \ref{SectionPreliminaries}). The point is that these classes are easy to characterise numerically without appealing to Bayer-Macr\`{\i}'s machinery. In this paper we try to provide a similar elementary description of the K\"{a}hler cone for $\mathrm{K3}$ and Kummer types varieties of arbitrary dimension which are deformations of the Hilbert scheme of points on a $\mathrm{K3}$ surface or Kummer variety, respectively.  Recall the definition.

\begin{definition}
Let $X$ be a $2n$-dimensional IHSM. We call it a $\mathrm{K3}$ type manifold if~$X$ is deformation equivalent to $\mathrm{Hilb}^n(S)$ for some $\mathrm{K3}$ surface $S.$ We call it a Kummer type manifold if~$X$ is deformation equivalent to $\mathrm{Kum}^n(A)$ for some complex two-dimensional  torus~$A.$ 
\end{definition}

So the goal of this paper is to give a simple and self-contained proof of the following theorems about the walls of the K\"{a}hler cone. Before we state these theorems, let us give some definitions. Let $X$ be an IHSM. Let $\alpha \in H^2(X,\mathbb{Z})$ be a primitive class. Let $q$ be the Beauville-Bogomolov form on~$X.$ We call the number 
$$
d(\alpha)=\min\{d>0 \mid q(\alpha,\beta)=d \; \text{for} \; \beta \in H^2(X,\mathbb{Z})  \}
$$
\noindent the \emph{divisibility} of $\alpha.$ By $\delta(\alpha)$ we denote the image of $\alpha/d(\alpha)$ by the natural homomorphism 
$$
H^2(X,\mathbb{Z})\otimes\mathbb{Q} \to H^2(X,\mathbb{Z})^*/H^2(X,\mathbb{Z}),
$$
\noindent  where the quotient $ H^2(X,\mathbb{Z})^*/H^2(X,\mathbb{Z})$  is the discriminant group (see Definition~\ref{definition:divisib}).

\begin{definition}
Let $X=\mathrm{Hilb}^n(B)$ (respectively, $X=\mathrm{Kum}^n(B)$) for a~$\mathrm{K3}$ surface~$B$ (respectively, for a complex two-dimensional  torus $B$). Let $\Delta$ be the exceptional divisor of the Hilbert--Chow morphism 
$$
\mathrm{Hilb}^n(B) \to \mathrm{Sym}^n(B).
$$
\noindent Then we call the divisor $\Delta/2$ (respectively, $\Delta \cap \mathrm{Kum}^{n-1}(B)/2$) \textit{half of the exceptional divisor.}

\end{definition}

\begin{Th}\label{TheoremK3}
Let $X$ be an IHSM of $\mathrm{K3}$ type of dimension $2n.$ Let $\alpha \in H^2(X,\mathbb{Z})$ be a primitive class and denote by~\mbox{$\widehat{\alpha} \in H_2(X,\mathbb{Z})$} the dual homology class of $\alpha.$ Denote by $q$ the Beauville--Bogomolov form on the lattice $H^2(X,\mathbb{Z})$ and its extension form on the dual lattice $H_2(X,\mathbb{Z})\hookrightarrow H^2(X,\mathbb{Z}) \otimes \mathbb{Q}.$ Let $d(\alpha)$ be the divisibility of $\alpha$ and~$\delta(\alpha)$ be the image of $\alpha/d(\alpha)$ in the discriminant group of the lattice~$H^2(X,\mathbb{Z}).$  Then $\alpha$ is MBM if and only if  there is an integer 
\begin{equation}\label{eq:bin0,n-1K3}
b \in [0,n-1]
\end{equation}
\noindent such that  
\begin{equation}\label{eq:q(alpha)K3}
q(\widehat{\alpha})=2a-\frac{b^2}{2(n-1)},
\end{equation}
\noindent  where $a$ is an integer number satisfied the inequalities
\begin{equation}\label{eq:-2<2a<bK3}
 -2 \leqslant 2a<\frac{b^2}{2(n-1)},
\end{equation}
\noindent and $\pm \delta(\alpha)=b.$ Moreover, the integer numbers $a$ and $b,$ satisfying  \eqref{eq:bin0,n-1K3}, \eqref{eq:q(alpha)K3} and~\eqref{eq:-2<2a<bK3}, uniquely determine the monodromy orbit of $\alpha.$

\end{Th}

In the proof of the theorem we obtain as a byproduct the following corollary:

\begin{Cor}\label{cor:TheoremK3}
Under the conditions of Theorem \ref{TheoremK3} we have:

\begin{enumerate}
\renewcommand\labelenumi{\rm (\arabic{enumi})}
\renewcommand\theenumi{\rm (\arabic{enumi})}

\item\label{1} If $b=1$ and $a=0,$ then the cohomological class  $\alpha$ is the class of half of the exceptional divisor on some deformation $Y=\mathrm{Hilb}^n(S)$ of $X$ with Picard group  
$$
\mathrm{Pic}(Y)=\mathrm{Pic}(S)\oplus \mathbb{Z}e,
$$
\noindent  where~$S$ is a $\mathrm{K3}$ surface and $e$ is a class of  half of exceptional divisor on $Y$.

\item\label{2} If $(a,b) \neq (0,1),$ then  the cohomological class $\alpha$ up to a rational multiple can be represented by
$$
2(n-1) \cdot x-be
$$
\noindent on some deformation $Y=\mathrm{Hilb}^n(S)$ of $X$ with Picard group 
$$
\mathrm{Pic}(Y)=\mathrm{Pic}(S) \oplus \mathbb{Z}e,
$$
\noindent such that $\mathrm{Pic}(S)=\mathbb{Z}x,$  $x^2=2a\geqslant -2$ and $e$ is a class of  half of the exceptional divisor on $Y.$

\end{enumerate}

\end{Cor}

\begin{Th}\label{TheoremKummer}
Let $X$ be an IHSM of Kummer type of dimension $2n\geqslant 4$. Assume that~\mbox{$\alpha \in H^2(X,\mathbb{Z})$} is a primitive class and denote by $\widehat{\alpha} \in H_2(X,\mathbb{Z})$ the dual homological class of $\alpha.$ Denote by $q$ the Beauville--Bogomolov form on the lattice~$H^2(X,\mathbb{Z})$ and its extension form on the dual lattice~$H_2(X,\mathbb{Z})\hookrightarrow H^2(X,\mathbb{Z}) \otimes \mathbb{Q}.$ Let $d(\alpha)$ be the divisibility of $\alpha$ and~$\delta(\alpha)$ be the image of~$\alpha/d(\alpha)$ in the discriminant group of the lattice $H^2(X,\mathbb{Z}).$   Then $\alpha$ is MBM if and only if  there is an integer number
\begin{equation}\label{eq:bin0,n-1Kummer}
b \in [1,n+1]
\end{equation}
\noindent such that 
\begin{equation}\label{eq:q(alpha)Kummer}
q(\widehat{\alpha})=2a-\frac{b^2}{2(n+1)},
\end{equation}
\noindent where~$a$ is an integer number satisfying the inequalities
\begin{equation}\label{eq:-2<2a<bKummer}
 0 \leqslant 2a<\frac{b^2}{2(n+1)},
\end{equation}
\noindent and $\pm \delta(\alpha)=b.$ Moreover, the integer numbers $a$ and $b,$ satisfying  \eqref{eq:bin0,n-1Kummer}, \eqref{eq:q(alpha)Kummer} and~\eqref{eq:-2<2a<bKummer}, uniquely determine the monodromy orbit of $\alpha.$

\end{Th}

Again, as in the $\mathrm{K3}$ case, in the proof of this theorem we also obtain as a byproduct the following corollary:

\begin{Cor}\label{cor:TheoremKummer}

Under the conditions of Theorem \ref{TheoremKummer} we have:

\begin{enumerate}
\renewcommand\labelenumi{\rm (\arabic{enumi})}
\renewcommand\theenumi{\rm (\arabic{enumi})}

\item\label{1'} If $b=1$ and $a=0,$ then the cohomological class $\alpha$  is the class of  half of the exceptional divisor on some deformation $Y=\mathrm{Kum}^n(A)$ of $X$ with Picard group 
$$
\mathrm{Pic}(Y)=\mathrm{NS}(A)\oplus \mathbb{Z}e,
$$
\noindent where~$A$ is a complex two-dimensional torus and $e$ is  a class of   half of the exceptional divisor on~$Y$.

\item\label{2'} If $(a,b) \neq (0,1),$ then  the cohomological class $\alpha$ up to a rational multiple can be represented by
$$
2(n+1) \cdot x-be
$$
\noindent on some deformation $Y=\mathrm{Kum}^n(A)$ of $X$ with Picard group 
$$
\mathrm{Pic}(Y)=\mathrm{NS}(A) \oplus \mathbb{Z}e,
$$
\noindent  such that $\mathrm{NS}(A)=\mathbb{Z}x,$ $x^2=2a\geqslant 0$ and $e$ is a class of   half of the exceptional divisor on $Y$.

\end{enumerate}

\end{Cor}

We also explain by direct computations how our results agree with that of A.~Bayer and E.~Macr\`{i}, and K.~Yoshioka.

\begin{Prop}\label{BMeqAV}
Let $X$ be a $\mathrm{K3}$ type manifold and $v$ be a primitive Mukai vector.  There is a one-to-one correspondence between  the set of orthogonal hyperplanes to MBM classes in the positive cone~\mbox{$\mathrm{Pos}(X),$} which were described in Theorem \ref{TheoremK3} and the set of walls in Theorem \ref{Theorem12.1}, given by
$$
a \in H^*_{alg}(X,\mathbb{Z}) \mapsto \mathrm{pr}_{v^{\perp}}(a)
$$
\noindent where $pr_{v^{\perp}}$ is a projection on $v^{\perp}.$

\end{Prop}
\begin{Prop}\label{BMeqAVabel}
Let $X$ be a Kummer type manifold and $v$ be a primitive Mukai vector. There is a one-to-one correspondence between  the set of orthogonal hyperplanes to MBM classes in the positive cone~\mbox{$\mathrm{Pos}(X),$} which were described in Theorem \ref{TheoremKummer} and the set of walls in Theorem \ref{Theorem12.1Kummer}, given by
$$
a \in H^*_{alg}(X,\mathbb{Z}) \mapsto \mathrm{pr}_{v^{\perp}}(a)
$$
\noindent where $pr_{v^{\perp}}$ is a projection on $v^{\perp}.$

\end{Prop}

\textbf{Notation.} Let $X$ be a K\"{a}hler manifold. If we need to fix some complex structure~$I$  on $X,$  then we write $(X,I)$ to define the manifold $X$ with complex structure~$I.$ By $H^{p,q}(X,I)$ we denote $(p,q)$-th component of Hodge decomposition of~$H^{p+q}(X,\mathbb{C})$ with respect to the complex structure $I.$

\vspace{5mm}

\textbf{Acknowledgment.} The author is grateful to E.~Yu.~Amerik for suggesting this problem, for a lot of discussions, for attention to this paper and for her patience. The author would like to thank the referee for carefully reading the manuscript and for giving such constructive comments which substantially helped improving the quality of the paper.

This work was performed at the Steklov International Mathematical Center and supported by the Ministry of Science and Higher Education of the Russian Federation (agreement no. 075-15-2022-265).  The author is a winner of the all-Russia mathematical August Moebius contest of graduate and undergraduate student papers and thanks the jury and the board for the high praise of his work.  The author was partially supported by Theoretical Physics and Mathematics Advancement Foundation “BASIS” and by the HSE University Basic Research Program.

	\section{Preliminaries}\label{SectionPreliminaries}
	
In this section we recall some theory about irreducible holomorphic symplectic manifolds.

\begin{definition}
A smooth compact  K\"{a}hler manifold $X$ is called \emph{an irreducible holomorphic symplectic manifold} (or just \emph{IHSM}) if it has trivial $\pi_1(X)$ and 
$$
H^0(X,\Omega_X^2) \simeq \mathbb{C}\sigma
$$
\noindent with a non-degenerate holomorphic 2-form $\sigma$.
\end{definition}

A few examples of such manifolds are known. Two series of examples were described by A.~Beauville in \cite{Beauville}. They are the Hilbert scheme of $n$ points on a $\mathrm{K3}$ surface and the  Kummer variety which is defined as follows.

\begin{definition}
Let $A$ be a complex two-dimensional  torus. Then the \textit{Kummer variety} of dimension $2n$  is the fibre over zero of the summation map
$$
s:\mathrm{Hilb}^{n+1}(A) \to A.
$$
\noindent We denote the Kummer variety of dimension $2n$ by $\mathrm{Kum}^n(A).$
\end{definition}

 \noindent Other two examples were found by O'Grady in \cite{OG1} and \cite{OG2}. All mentioned series of examples are not pairwise birational, because they have different second Betti numbers.

\bigskip

\subsection*{Beauville--Bogomolov form and Mukai lattice.} The IHS manifolds have an interesting feature  which, in general, only surfaces have. If~$X$ is an IHSM, then there is a non-degenerate quadratic form $q$ on the second integral cohomology~$H^2(X,\mathbb{Z})$ which is called \emph{Beauville--Bogomolov form}. This form is of signature $(3,b_2(X)-3)$ and its restriction on $H^{1,1}(X,\mathbb{R})$ is of signature~\mbox{$(1,b_2(X)-3).$} This form is uniquely up to multiplication by a non-zero constant.

\begin{Th}[{\cite[Theorem 4.7]{F} and~\cite[Proposition 23.14]{GJH}}]\label{th:beaubilebogomolov}
Let $X$ be a IHSM of dimension $2n.$ Denote by~$q$ the Beauville--Bogomolov form. Then there is a constant~\mbox{$c \in \mathbb{R}_{>0}$} depending only on the manifold $X$ such that for $\alpha \in H^2(X,\mathbb{Z})$ we get 
\begin{equation}\label{eq:integral}
\int_{X}\alpha^{2n}\,=\,cq(\alpha)^n.
\end{equation}
\noindent  In particular, $q$ can be renormalized such that $q$ is a primitive integral quadratic form on $H^2(X,\mathbb{Z}).$
\end{Th}

In fact, the primitive integral Beauville-Bogomolov form is uniquely defined by~\eqref{eq:integral} and signature on $H^{1,1}(X,\mathbb{R}).$ This means that the quadratic form $q$ depends only on topology of the manifold (see~\cite[Remark 23.15]{GJH}). Thus, it is invariant under the deformations. We always assume that the Beauville-Bogomolov form is primitive integral form on $H^2(X,\mathbb{Z}).$

\begin{Remark}\label{homology}
The second cohomology group of an IHSM equipped with Beauville--Bogomolov form is a lattice. We have the embedding $H_2(X,\mathbb{Z}) \hookrightarrow H^2(X,\mathbb{Z}) \otimes \mathbb{Q}$ which is given by the lattice duality. 
\end{Remark}

Let us give the description of the Beauville--Bogomolov form for the Hilbert scheme of points on a $\mathrm{K3}$ surface and the Kummer variety. First of all, we recall the definition and some properties of the Mukai lattice, which will be extremely important below. Let 
$$
L\,=\,H^0(X,\mathbb{Z}) \oplus H^2(X,\mathbb{Z}) \oplus H^4(X,\mathbb{Z})
$$
\noindent be the lattice of even cohomologies of a surface $X$ with the quadratic form which is called Mukai pairing and given by
$$
((r_1,\alpha_1,s_1),(r_2,\alpha_2,s_2))\,=\,\alpha_1 \cdot \alpha_2 - r_1 \cdot s_2 - r_2\cdot s_1,
$$

\noindent where $(r_i,\alpha_i,s_i) \in L$ such that $r_i \in H^0(X, \mathbb{Z}),$ $\alpha_i \in H^2(X, \mathbb{Z})$ and $s_i \in H^4(X, \mathbb{Z}).$  Let
$$
H^*_{alg}(X,\mathbb{Z})\,=\,H^0(X,\mathbb{Z}) \oplus \mathrm{NS}(X) \oplus H^4(X,\mathbb{Z}).
$$
\noindent It has the following geometric meaning: for the Grothendieck group of the category of coherent sheaves $\mathcal{K}_0(X)$ we have a surjective homomorphism 
$$
v: \mathcal{K}_0(X) \twoheadrightarrow H^*_{alg}(X)
$$
\noindent which maps the class of a sheaf $\mathcal{F}$ to the Mukai vector  $v(\mathcal{F})\,=\,ch(\mathcal{F})\sqrt{td(X)},$  where~$ch(\mathcal{F})$ is the Chern character and $td(X)$ is the Todd class of tangent bundle of $X$. In the case  of a $\mathrm{K3}$ surface we get
$$
v(\mathcal{F})\,=\,(rk(\mathcal{F}),c_1(\mathcal{F}),\chi(\mathcal{F})-rk(\mathcal{F}))
$$
 \noindent and for a complex two-dimensional  torus we have
 $$
 v(\mathcal{F})\,=\,(rk(\mathcal{F}),c_1(\mathcal{F}),\chi(\mathcal{F})),
 $$
	
\noindent where $rk(\mathcal{F})$ is the rank, $c_1(\mathcal{F})$ is the first Chern class and 
$$
\chi(\mathcal{F})\,=\,\sum_i (-1)^i \dim \mathrm{Ext}^i(\mathcal{F},\mathcal{F})
$$
\noindent  is the Euler characteristic.	

The Hilbert scheme of $n$ points on the  $\mathrm{K3}$ surface $S$  is a moduli space of ideal sheaves of $n$ points on $S$ with respect to a stability condition $\sigma$. These are sheaves with the Mukai vector $v\,=\,(1,0,-n+1).$  The orthogonal complement of $v$ with respect to the Mukai quadratic form is isomorphic to the N\'{e}ron-Severi group of~$M_{\sigma}(v).$ So we have
$$
v^{\perp} \simeq \mathrm{NS}(S) \oplus \mathbb{Z} e,
$$
	
\noindent where $e\,=\,(1,0,n-1)$ and $e^2\,=\,-2(n-1).$ This means that the restriction of the Beauville--Bogomolov form for $\mathrm{Hilb}^n(S)$ on $\mathrm{NS}(S)$ coincides with the intersection form on $\mathrm{NS}(S),$ and~$e$ is orthogonal to  $\mathrm{NS}(S).$

The same description works for a complex two-dimensional  torus $A$. In this case we consider the moduli space $M_{\sigma}(v)$ for the Mukai vector~\mbox{$v\,=\,(1,0,-n-1).$} Every sheaf $\mathcal{F} \in M_{\sigma}(v)$ can be written as $\mathcal{F}=\mathcal{L} \otimes I_{T},$ where~\mbox{$\mathcal{L} \in \mathrm{Pic}^0(A)=A^{\vee}$} and~\mbox{$T \in \mathrm{Hilb}^{n+1}(A).$} The generalized Kummer variety $\mathrm{Kum}^n(A)$ is the fibre over zero of the Albanese map 
\begin{align*}
M_{\sigma}(v) & \longrightarrow A \times A^{\vee}\\
\mathcal{F} &\longmapsto \left(\mathrm{Alb}(T),\mathcal{L}\right)
\end{align*}

\noindent where $\mathrm{Alb}:\mathrm{Hilb}^{n+1}(A) \to A$ is a summation map.  We have the isomorphism
$$
v^{\perp} \simeq \mathrm{NS}(A) \oplus \mathbb{Z} e,
$$

\noindent where $e=(1,0,n+1)$ and $e^2=-2(n+1).$ This means that the restriction of the Beauville--Bogomolov form for $\mathrm{Kum}^n(A)$ on~$\mathrm{NS}(A)$ coincides with the intersection form on~$\mathrm{NS}(A)$ and~$e$ is orthogonal to  $\mathrm{NS}(A).$

The vector $e$ has the following geometric description in both $\mathrm{K3}$ and Kummer cases.

\begin{Prop}
Let $\Delta$ be the exceptional divisor of the Hilbert--Chow morphism 
$$
\mathrm{Hilb}^n(B) \to \mathrm{Sym}^n(B),
$$
\noindent where $B$ is either $\mathrm{K3}$ surface, or a complex two-dimensional torus. If $B$ is a $\mathrm{K3}$ surface and $e=(1,0,n-1),$ then the class of $\Delta$ is $2e.$  If $B$ is a complex two-dimensional torus and $e=(1,0,n+1),$ then the class of $\Delta \cap \mathrm{Kum}^{n-1}(B)$ is $2e.$
\end{Prop}

\begin{proof}
The proposition follows directly from~\cite[Remark after Proposition 6]{Beauville} for $\mathrm{K3}$ case and \cite[Chapter 2, \S1, Remark 5]{Britze} for Kummer case.
\end{proof}

\begin{definition}
 We denote by $\widehat{e}$ the class of a general fibre on the exceptional divisor of the Hilbert--Chow morphism in $H_2(X,\mathbb{Z}).$ It satisfies the equation 
$$
q(e,\widehat{e})=-1.
$$
\end{definition}

\begin{Remark}
Note, that if $X=\mathrm{Hilb}^n(B)$ (respectively, $X=\mathrm{Kum}^n(B)$) for a~$\mathrm{K3}$ surface (respectively, for a complex two-dimensional  torus $B$),  then we have the following isomorphisms:
\begin{gather*}
H^2(X,\mathbb{Z}) \simeq H^2(B,\mathbb{Z}) \oplus \mathbb{Z}e;\\
\mathrm{Pic}(X) \simeq \mathrm{NS}(B) \oplus \mathbb{Z}e;\label{Picard}\\
\mathrm{N}_1(X) \simeq \mathrm{N}_1(B) \oplus \mathbb{Z}\widehat{e},
\end{gather*}
\noindent where by $\mathrm{N}_1$ we denote the group of $1$-cycles modulo numerical equivalence. 
\end{Remark}

\bigskip

\subsection*{The space of deformations.} Now we recall some definitions and facts of deformation theory.

\begin{definition}
Let $X_1$ and $X_2$ be complex manifolds. We say that they are \textit{deformation equivalent} if there is a smooth proper holomorphic morphism $\mathcal{X} \to B$ such that $B$ is connected and there are two points $b_1,b_2 \in B$ such that $\mathcal{X}_{b_1} \simeq X_1$ and $\mathcal{X}_{b_2} \simeq X_2.$ Here the variety $B$ is called \textit{a space of deformations.}

\end{definition}

Now we give the description of the deformation space of IHSM which will be convenient for us.

\begin{definition}
Let $X$ be an IHSM, $\text{Diff}_0(X)$ be a connected component of its diffeomorphism group, $\text{Comp}$ be the space of complex structures of the K\"{a}hler type on $X.$ Then $\text{Teich}\,=\,\text{Comp}/\text{Diff}_0(X)$ is the \textit{Teichm\"{u}ller space}.
\end{definition}

Note that the Teichm\"{u}ller space is non-Hausdorff. However, there is an important result which is given in~\cite{Huybrechts1999}. 

\begin{Th}[{\cite[Theorem 4.3]{Huybrechts1999}}]\label{th:birationalnonsep}
Let $X$ and $X'$ be IHSMs which are diffeomorphic. Assume that they are inseparable in the space $\mathrm{Teich}.$ Then $X$ and $X'$ are birational.
\end{Th}

It is convenient to consider the \textit{birational Teichm\"{u}ller space}  $\text{Teich}_b=\text{Teich}/ \sim,$ where~\mbox{$X \sim X'$} if and only if they are inseparable in $\mathrm{Teich}.$ The  birational Teichm\"{u}ller space is the Hausdorff reduction of the  Teichm\"{u}ller space. By Theorem~\ref{th:birationalnonsep} every point of $\mathrm{Teich}$ over any~$x \in \text{Teich}_b$ corresponds to some birational model of~$X.$

\bigskip

\subsection*{Cones associated to an irreducible holomorphic symplectic manifold.} Now let us recall some facts about the positive cone for IHSM.

\begin{definition}
Let $X$ be an IHSM. Both the set of $\alpha \in \mathrm{NS}(X)_{\mathbb{R}}$ for projective $X$ and the set of $\alpha \in H^{1,1}(X)_{\mathbb{R}}$ for K\"{a}hler $X,$  such that $q(\alpha)>0,$ have two connected components. The positive cone  is the component that contains ample or  K\"{a}hler classes, respectively. We denote both of them by $\mathrm{Pos}(X)$ when there is no confusion, otherwise we denote by $\mathrm{Pos}_{\text{amp}}(X)$ and $\mathrm{Pos}_{\text{K\"{a}h}}(X)$ the positive cones of ample and  K\"{a}hler classes, respectively.

\end{definition}

\begin{Remark}
Note that  $\mathrm{Pos}_{\text{amp}}(X)$ and $\mathrm{Pos}_{\text{K\"{a}h}}(X)$ are different, because in general the inclusion
$$
\mathrm{NS}(X)_{\mathbb{R}}=H^{1,1}(X,\mathbb{Z})_{\mathbb{R}} \subset H^{1,1}(X,\mathbb{R})
$$
\noindent is strict.  However, the intersection of $\mathrm{Pos}_{\text{K\"{a}h}}(X)$ with $\mathrm{NS}(X)_{\mathbb{R}}$ gives us $\mathrm{Pos}_{\text{amp}}(X)$. This is the reason why we will mainly deal with the positive cone $\mathrm{Pos}_{\text{K\"{a}h}}(X).$ 
\end{Remark}

\begin{definition}
Let $X$ be a K\"{a}hler manifold. The K\"{a}hler cone 
$$
\mathcal{K}(X) \subset H^{1,1}(X,\mathbb{R})
$$
\noindent is an open convex cone of K\"{a}hler classes. 

\end{definition}

\begin{definition}
Let $X$ be a projective manifold. The ample cone 
$$
\mathrm{Amp}(X) \subset \mathrm{NS}(X)_{\mathbb{R}}
$$
\noindent  is the open cone of finite sums $\sum_ia_iL_i,$ where $a_i \in \mathbb{R}_{>0}$ and $L_i$ are ample classes in $\mathrm{NS}(X).$  The nef cone $\mathrm{Nef}(X) \subset \mathrm{NS}(X)_{\mathbb{R}}$ is the set of classes $\mathrm{NS}(X)_{\mathbb{R}}$ such that their intersections with curves on $X$ is non-negative.
\end{definition}

The closure of the ample cone is the nef cone, i.e. we have
$$
\mathrm{Amp}(X)  = \mathrm{Int}\mathrm{Amp}(X) \subset \mathrm{Nef}(X) = \overline{\mathrm{Amp}(X)},
$$

\noindent where $\mathrm{Int}\mathrm{Amp}(X)$ is the interior of the ample cone of $X.$ We have the following result.

\begin{Th}[{\cite{Bou}}]\label{ThofBoucksom} The K\"{a}hler cone of an IHSM $X$ consists of all elements $\alpha$ in~$\mathrm{Pos}(X)$ such that $q(\alpha,C) \,>\, 0$ for all rational curves $C$.
\end{Th}

According to Remark \ref{homology}, one may view homology classes as cohomology classes with rational coefficients, so that $q(\alpha,C)$ is a well-defined rational number. Theorem~\ref{ThofBoucksom} is a generalization of the description of the K\"{a}hler cone for a $\mathrm{K3}$ surface. In the two-dimensional case the positivity of the intersection with $(-2)$-curves  defines ample or K\"{a}hler classes. However, in the higher-dimensional case the situation is more complicated. By Theorem \ref{ThofBoucksom}, the walls of the K\"{a}hler cone are given by the orthogonal elements to the classes of rational curves, but now the value of the quadratic form can be different on different rational curves. 

Nevertheless, it is clear that for  $\alpha \in \mathrm{Pos}(X)$ to check that it lies in the  K\"{a}hler cone we only need to check the positivity of $q(\alpha,C)$ on rational curves $C$ with negative square. Indeed, as Beauville--Bogomolov form has signature $(1,b_2-3)$ on~$(1,1)$-classes we get that the orthogonal hyperplane to a curve with non-negative square does not intersect the positive cone.

\bigskip

\subsection*{MBM classes.} A rational curve $C$ is called \textit{extremal} if its cohomological class $[C]$ cannot be written as $[C]=Z_1+Z_2,$ where~$Z_1$ and $Z_2$ are two non-proportional effective cycles in $\mathrm{NS}(X).$  By Theorem~\ref{ThofBoucksom}, the walls of the  K\"{a}hler cone are given by the orthogonal hyperplanes to the extremal rational curves. It turns out that the chamber decomposition of the positive cone is given by the orthogonal hyperplane to the class with negative Beauville--Bogomolov square represented by a rational curve on some deformation of $X.$ We recall some definitions and results about these classes.

\begin{definition}[{\cite[Definition 1.5]{MarkmanK3}}]
Let $X$ be an IHSM. An automorphism $g$ of~$H^2(X,\mathbb{Z})$ is called a \textit{monodromy operator}, if there exists a family
$$
\pi :\mathcal{X} \to B
$$
\noindent of IHSM such that $\psi:X \xrightarrow{\sim} \mathcal{X}_{b_0}$ for some $b_0 \in B$ and $g$ belongs to the image of~$\pi_1(B,b_0)$ under the monodromy representation. In other words, the automorphism $g \in \mathrm{Aut}(H^2(X,\mathbb{Z}))$  is a monodromy operator if there is a loop $\gamma: [0,1] \to B$ such that $\gamma(0)=\gamma(1)=b_0$ and the parallel transport in the local system $R^2\pi_*\mathbb{Z}$ along $\gamma$ induces the homomorphism 
$$
\psi_* \circ g \circ \psi^*:H^2(\mathcal{X}_{b_0},\mathbb{Z}) \to H^2(\mathcal{X}_{b_0},\mathbb{Z}).
$$

 The monodromy group~\mbox{$\mathrm{Mon}(X)$} is a subgroup of $\mathrm{Aut}(H^2(X,\mathbb{Z}))$ generated by all monodromy operators. By~\cite[Theorem~7.2]{Verbit} the subgroup $\mathrm{Mon}(X)$ coincides with the natural image of the mapping class group~$\Gamma=\mathrm{Diff}(X)/\mathrm{Diff}_0(X)$ in~$\mathrm{Aut}(H^2(X,\mathbb{Z})),$ where $\mathrm{Diff}(X)$ is the group of diffeomorphisms of the manifold~$X$ and~$\mathrm{Diff}_0(X)$ is its connected component containing the identity, i.e. the group of isotopies.
\end{definition}

\begin{definition}[{\cite[Definition 1.13]{AV1401}}]\label{MBM1401}
A negative rational class 
$$
\alpha \in H^{1,1}(X,I) \cap H^2(X,\mathbb{Q})
$$
\noindent  on $(X,I)$ is called \emph{MBM} if for some $\gamma \in \mathrm{Mon}(X)$ the subspace $\gamma(\alpha)^{\perp} \subset H^{1,1}(X,I)$ contains a wall of the K\"{a}hler cone of a birational model $(X,I')$ of $(X,I).$
\end{definition}

By Remark~\ref{homology} it makes sense to speak of $(1,1)$-classes represented by curves. 

\begin{definition}
Let $X$ be an IHSM and let us fix the K\"{a}hler class $\omega.$ Let $U \subset X$ be an irreducible uniruled subvariety. A rational curve $C \subset U$ is called \textit{minimal} if~$\int_C \omega$ is minimal among all curves in $U$ passing through a general point of $U.$
\end{definition}

\begin{definition}
Let $X$ be an IHSM.  A rational curve $C \subset X$ is called \textit{MBM} if it is minimal and its cohomology class $[C] \in H^2(X,\mathbb{Z})$ has negative Beauville--Bogomolov square.
\end{definition}

There is a convenient criterion that allow to check the minimality of a rational curve on IHSM.

\begin{Th}[{\cite[Corollaries 4.5 and 4.6]{AV1401} and~\cite[Remark 1.2]{AV21}}]\label{th:minimaliff2n-2}
Let $X$ be an IHSM of dimension $2n$ and let  $C \subset X$ be a  rational curve. Then $C$ is minimal if and only if $C$ deforms in~$(2n-2)$-parameter family.
\end{Th}

However we will use the following criterion  that allow to check the minimality of a rational curve on IHSM. First of all, recall the definition.

\begin{definition}
Let $Z$ be a K\"{a}hler manifold covered by rational curves. Let 
$$
r \colon Z \dashrightarrow Q
$$
\noindent be an almost holomorphic fibration such that its general fibre is rationally connected. Then the fibration $r$ is called \textit{the rational quotient } of $Z.$

\end{definition}

\begin{Th}[{\cite[Theorem 4.4]{AV1401}}]\label{th:criteriaminimality}
Let $X$ be an IHSM and $C \subset X$ be a rational curve. Let $Z$ be an irreducible component of the locus covered by the deformations of~$C$ on $X.$ Then $Z$ is a coisotropic subvariety of $X.$ The fibres of the rational quotient of the desingularization of $Z$ have dimension equal to the codimension of~$Z$ in $X.$

\end{Th}

The following proposition gives the convenient description of MBM classes, which will be used in the classification of MBM classes on $\mathrm{K3}$ and Kummer types manifolds (see Proposition 2.1 and the discussion after Proposition 2.2 in~\cite{AVsurvey}).

\begin{Prop}\label{prop:convenientdefinitionofMBM}
 The class $\alpha \in H^{1,1}(X,I) \cap H^2(X,\mathbb{Q})$ is  the MBM class if and only if its dual class $\widehat{\alpha} \in H_2(X,\mathbb{Q})$  is represented, up to rescaling, by a minimal  rational curve on some deformation of $X.$
\end{Prop}
\begin{Prop}[{\cite[Corollary 5.13]{AV1401}}]\label{prop:simultaneously} A negative class $\alpha \in H^2(X,\mathbb{Z})$ is MBM or not simultaneously in all complex structures, where it is of type (1,1).
\end{Prop}

\begin{Remark}\label{remark:vanishingofcurve}
By definition, an MBM class $\alpha$ on an IHSM $X$ up to a sign is represented by a curve on a sufficiently general deformation $X'$ of $X.$ This means that it must be represented by a possibly reducible curve on all deformations.

\end{Remark}

\begin{definition}
Let $X$ be an IHSM. Let $C$ be an MBM curve and let $B$ be the space of all deformations of $C.$  We call the \textit{MBM locus} of $C$ the union of all curves parametrized by $B.$
\end{definition}

\begin{definition}
Let $X$ be an IHSM. Let $\alpha \in H^2(X,\mathbb{Z})$ be an MBM class. Then we call the \textit{MBM locus} of the class $\alpha$ the union of all MBM loci of MBM curves $C$ such that $[C]$ is proportional to $\alpha.$ 
\end{definition}

\begin{Remark}\label{remark:criteriaminimality}
From Theorems~\ref{th:minimaliff2n-2} and~\ref{th:criteriaminimality} and \cite[Remark 1.2]{AV21} we get that to prove the minimality of the rational curve $C$ in IHSM $X$ one should find MBM locus $Z$ of $C$ and prove that it is a rational quotient  such that its fibres  of the desingularization of $Z$ have dimension equal to the codimension of~$Z$ in $X$ and the base of this fibration is not unirational.

\end{Remark}

Now we recall the theorem which states that MBM loci on the deformation equivalent IHSM  are real analytic isomorphic.

\begin{Th}[{\cite[Theorem 1.8]{AV21}}]\label{th:analiticallyequiv}
Let $X$ be an IHSM of $\mathrm{K3}$ or Kummer type and let $\alpha \in H^2(X,\mathbb{Q})$ be MBM and $T$ be the MBM locus on $X$ of the class $\alpha.$ Let $X'$ be a deformation of $X$ in $\mathrm{Teich}$ such that $\alpha$ is $(1,1)$ and is the class of a minimal rational curve on both $X$ and $X'.$ Let $T'$ be an MBM locus on $X'$ of $\alpha.$ Then $T$ and~$T'$ are real analytic isomorphic.
\end{Th}

Now we recall the definition of the period map and the theorem about chamber decomposition of the positive cone.

\begin{definition}
Let $X$ be an IHSM. Denote by $\text{Per}$ the map 
 $$
\mathrm{Per}:\mathrm{Teich} \to \mathbb{P}H^2(X,\mathbb{C})
$$
\noindent such that $\text{Per}(I)=H^{2,0}(X,I) \in \mathbb{P}H^2(X,\mathbb{C})$ for a complex structure $I \in \text{Teich}.$ The image of $\text{Per}$ is %an open subset in 
$$
\mathbb{P}\text{er}=\{l \in \mathbb{P}H^2(X,\mathbb{C}) \mid q(l)=0,\; q(l+\bar{l})>0 \}.
$$
\end{definition}

\begin{definition}
Let $\alpha$ be an MBM class. The space $\text{Teich}_{\alpha}$ consists of complex structures $I$ such that $\alpha$ is of type (1,1) with respect to $I.$ We have 
$$
\text{Teich}_{\alpha}=\text{Per}^{-1}(\alpha^{\perp}).
$$
\end{definition}

\begin{Th}[{\cite[Theorem 1.19]{AV1401}, \cite[Theorem 1.4]{Markman}}]
Let $X$ be an IHSM and~$I$ be its complex structure. Then the points in~\mbox{$\textup{Per}^{-1}\left(\textup{Per}(I)\right)$} are in  one-to-one correspondence with the K\"{a}hler chambers, i.e. the connected components of 
$$
\mathcal{BK}_X=\mathrm{Pos}(X) \setminus \bigcup_{z \in \textup{MBM classes}} z^{\perp}.
$$ 
\end{Th}

\noindent This means that the connected components (or chambers) of $\mathrm{Pos}(X)$ are K\"{a}hler (or ample) cones of $X$ and its birational models, as well as their monodromy images (see \cite{AV} and other papers of the authors). Here we use the following theorem:

\begin{Th}[{\cite[Theorem 1.3]{Markman}}]
Let $X$ and $Y$ be IHSM which are deformation equivalent. Then they are bimeromorphic if and only if there is a parallel transport operator $f:H^2(X,\mathbb{Z}) \to H^2(Y,\mathbb{Z})$ which is an isomorphism of integral Hodge structures.
\end{Th}

 So our main goal is to find numerical characterization for these MBM classes for Beauville examples. For more details about the Teichm\"{u}ller space see~\mbox{\cite[Section~3]{AV}.}

\bigskip

\subsection*{Torelli theorem.} For IHSM there is a generalization of Global Torelli Theorem for $\mathrm{K3}$ surfaces, namely 

\begin{Th}[{\cite[Theorem 1.18]{Verbit}}]\label{Torelli}
For each connected component of the birational Teichm\"{u}ller space $\mathrm{Teich}_b^0,$  the period map $\textup{Per}: \textup{Teich}_b^0 \to \mathbb{P}\textup{er}$ is a diffeomorphism.
\end{Th}

\begin{Cor}\label{TorelliCor}
Let $(X,I)$ be an IHSM with a complex structure $I.$ Consider the lattice $L=H^2(X,\mathbb{Z})$ equipped with the Beauville--Bogomolov form. Suppose that 
$$
\alpha \in L \cap H^{1,1}(X,I)
$$
\noindent is a negative vector, i.e. $q(\alpha)<0,$ and~\mbox{$\alpha \in L' \subset L,$} where $L'$ is a primitive sublattice of $L$ and $(L'^{\perp})_{\mathbb{R}}$ contains a positive $2$-plane. Then there is a deformation $(X,I')$ of~$(X,I)$ in $\mathrm{Teich}_{\alpha}$ such that $\mathrm{Pic}(X,I')=L'.$
\end{Cor}

\begin{proof}
Let $M$ be a positive $2$-plane in $L'^{\perp}.$ It is not hard to see that we can find a vector $l \in M \otimes \mathbb{C}$ such that $q(l)=0$ and $q(l+\bar{l})>0.$  By Theorem~\ref{Torelli} there is a complex structure $I' \in \mathrm{Teich}_{\alpha}$ such that 
$$
\mathrm{Per}(I')=l \in  \mathbb{P}H^2(X,\mathbb{C}).
$$

\noindent If such a positive $2$-plane $M$ is generic, then 
$$
L \cap H^{1,1}(X,I')=L',
$$
\noindent but as $\mathrm{Pic}(X,I')=L \cap H^{1,1}(X,I'),$ we get the desired.

\end{proof}

\begin{Cor}\label{Corollary:deformation}
Let $(X,I)$ be a $\mathrm{K3}$ (respectively,~Kummer) type manifold and let
$$
L=H^2(X,\mathbb{Z})
$$
\noindent be the lattice equipped with the Beauville--Bogomolov form. Suppose that there is a primitive~\mbox{$\alpha \in L$} such that $\alpha \in H^{1,1}(X,I)$ and~\mbox{$q(\alpha)<0.$} Then the manifold~$X$ can be deformed to $Y$ in $\textup{Teich}_{\alpha}$ such that $Y=\mathrm{Hilb}^n(S)$ (respectively,~$Y=\mathrm{Kum}^n(A)$), where~$S$ is a $\mathrm{K3}$ surface with cyclic Picard group (respectively, $A$ is a complex two-dimensional  torus with cyclic N\'{e}ron-Severi group).
\end{Cor}

\begin{proof}

Let $e$ be  a class of half of the exceptional divisor on some deformation $Y$ of $X$ which is the Hilbert scheme of a~$\mathrm{K3}$ surface (respectively,~Kummer manifold). Note that the lattice $\langle \alpha, e \rangle$ is not positive definite, since both $q(\alpha)$ and $q(e)$ are negative. Thus, applying Corollary~\ref{TorelliCor} to the lattice~\mbox{$\langle \alpha, e \rangle$} we obtain a deformation $Y'$ of~$X,$ such that 
$$
\mathrm{Pic}(Y')=\langle \alpha, e \rangle.
$$
\noindent Such a manifold $Y'$ has the same period point in $\mathbb{P}\mathrm{er}$ as the Hilbert scheme of points on a  $\mathrm{K3}$ surface (respectively, the Kummer variety). Indeed, the period points of  Hilbert schemes (respectively,~Kummer varieties) in $\mathbb{P}\mathrm{er}$ form the union of the orthogonal hyperplanes to $e$ and to its monodromy images.  Hence $Y'$ is unseparable from Hilbert scheme~$Y=\mathrm{Hilb}^n(S)$ for some $\mathrm{K3}$ surface~$S$ (respectively, unseparable from Kummer variety $Y=\mathrm{Kum}^n(A)$ for some complex two-dimensional  torus~$A$) in the Teichm\"{u}ller space. In other words, we can deform $Y'$ to $Y$ which is the Hilbert scheme of points on a $\mathrm{K3}$ surface (respectively,~the Kummer variety) such that $Y$ and $Y'$ have the same Picard group.

\end{proof}

\bigskip

\subsection*{Strategy of the proof of Theorems~\ref{TheoremK3} and \ref{TheoremKummer}.} Let us consider the Teichm\"{u}ller space of $\mathrm{K3}$ or Kummer type manifold. Now we briefly explain our strategy for proof of Theorems \ref{TheoremK3} and \ref{TheoremKummer}.

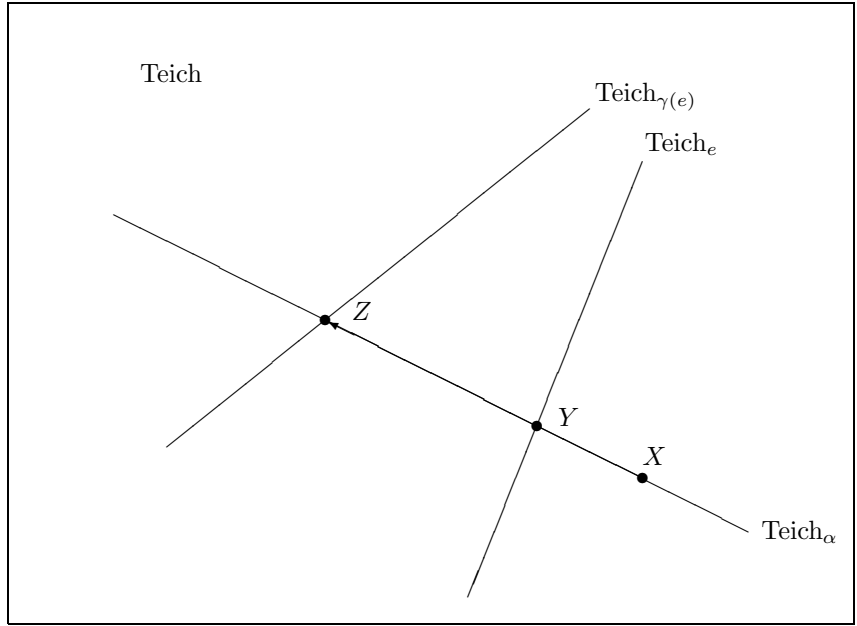
\begin{figure}
\begin{picture}(320,250)
\put(40,160){\line(2,-1){240}}
\put(240,180){\line(-2,-5){66}}
\put(220,200){\line(-5,-4){160}}
\put(240,60){\circle*{4}}
\put(200,80){\circle*{4}}
\put(120,120){\circle*{4}}
\put(240,60){\vector(-2,1){119}}
\put(285,37){$\text{Teich}_{\alpha}$}
\put(241,184){$\text{Teich}_{e}$}
\put(222,203){$\text{Teich}_{\gamma(e)}$}
\put(240,65){$X$}
\put(208,80){$Y$}
\put(130,120){$Z$}
\put(50,210){$\text{Teich}$}
\put(0,5){\line(1,0){320}}
\put(0,5){\line(0,1){235}}
\put(0,240){\line(1,0){320}}
\put(320,5){\line(0,1){235}}
\end{picture}
\caption{Teichm\"{u}ller space of IHSM $X.$}
\label{figure:teich}
\end{figure}

Let $(X,I)$ be an IHSM of $\mathrm{K3}$ (respectively, Kummer) type with a complex structure $I.$ Let 
$$
\alpha \in H^{1,1}(X,I) \cap H^2(X,\mathbb{Q})
$$
\noindent be a class of negative Beauville--Bogomolov square and $e \in H^2(X,\mathbb{Z})$ be a class of  half of the exceptional divisor. Figure~\ref{figure:teich} shows the Teichm\"{u}ller space of $X.$  The points correspond to the complex structures in $\text{Teich}.$ The point $X$ on Figure~\ref{figure:teich} corresponds to the complex structure of the original $X.$ 

In Step $1$ of the proof by Corollary~\ref{Corollary:deformation} we deform $X$ to $Y=\mathrm{Hilb}^n(S')$ for a $\mathrm{K3}$ surface $S'$ in the~\mbox{$\mathrm{K3}$} case or to $Y=\mathrm{Kum}^n(A')$ for a complex two-dimensional  torus~$A'$ in the Kummer case. In the picture the manifold $Y$ corresponds to some point which lies in the intersection of $\text{Teich}_{\alpha}$ and~\mbox{$\text{Teich}_{e}.$}

A general point of $\text{Teich}_{\alpha} \cap \text{Teich}_{e}$ does not lie on other $\text{Teich}_{\beta}$ for~\mbox{$\beta \in H^2(X,\mathbb{Z}),$} thus, we have $\text{Pic}(Y)=\langle \alpha, e\rangle.$ Hence we get that $\mathrm{Pic}(S')$ (respectively, $\mathrm{NS}(A')$) generated by  unique element.

In Step $2$ of the proof we deform $Y$ to $Z$ such that $Z$ lies in the intersection $\text{Teich}_{\alpha} \cap \text{Teich}_{\gamma(e)}$ and does not lie on other $\mathrm{Teich}_{\beta}$ for $\beta \in H^2(X,\mathbb{Z}),$ and the element $\gamma \in \text{Mon}(X)$ is chosen with certain restrictions which will be explained below. The manifold $Z$ is diffeomorphic to the variety $\gamma^{-1}(Z)$ with Picard group~\mbox{$\mathrm{Pic}(\gamma^{-1}(Z))=\langle  \gamma^{-1}(\alpha), e \rangle.$ } By  Corollary~\ref{TorelliCor} there is the Hilbert scheme of points on a $\mathrm{K3}$ surface (respectively, Kummer variety) with the same Picard lattice as~$\gamma^{-1}(Z)$ has. We will do all calculations on  the Hilbert scheme on points on $\mathrm{K3}$ surface (respectively, Kummer variety) as MBM property is invariant under deformations by Proposition~\ref{prop:simultaneously}.  

In Steps $3$ and $4$ we consider a rational curve corresponding to $\gamma^{-1}(\alpha)$ on $Z$ and give necessary and sufficient conditions on this rational curve to be MBM. 

In Step $5$ we prove that the description of MBM classes given in Theorems \ref{TheoremK3} and \ref{TheoremKummer} is full.

In Step $6$ we show that the numerical condition on $\alpha$  given in Theorems \ref{TheoremK3} and~\ref{TheoremKummer}  uniquely determines its monodromy orbit.

\section{Lattices  of $\mathrm{K3}$ and Kummer types manifolds}

In this section we collect some basic facts about lattices.

\begin{definition}\label{definition:divisib}
Let $L$ be a lattice and $L^{\vee}$ be its dual lattice, i.e.
$$
L^{\vee}=\{l^* \in L \otimes \mathbb{Q} \mid (l^*,l) \in \mathbb{Z} \; \text{for all}\; l \in L\}.
$$
\noindent Then the quotient group $L^{\vee}/L$ is called the discriminant group. For an element~\mbox{$l \in L$} denote by $d(l)$ the divisibility of $l,$ i.e. the positive integer number such that
$$
(l,L)=d(l)\mathbb{Z}.
$$
\noindent For $l \in L$ denote by $\delta(l)$  the image of $l/d(l)$ in the  discriminant group $L^{\vee}/L.$
\end{definition}

\begin{definition}
Let $L$ be a lattice. If $\alpha \in L$ is a primitive element, then the element 
$$
\widehat{\alpha}=\frac{\alpha}{d(\alpha)} \in L^{\vee}
$$
\noindent  is called the dual element to $\alpha.$
\end{definition}

\begin{Example}\label{ex:K3lattice}
The lattice of the second cohomology group for~$\mathrm{K3}$ type manifolds of dimension $2n$ is $L=U^{\oplus 3} \oplus E(-8)^{\oplus 2} \oplus \langle -2(n-1) \rangle.$ Thus, we have
$$
L^{\vee}/L \simeq \mathbb{Z}/{2(n-1)}\mathbb{Z}.
$$
\end{Example}

\begin{Example}\label{ex:Kummerlattice}
Similarly, the lattice of the second cohomology group for Kummer type manifolds of dimension~$2n$ is $L=U^{\oplus 3} \oplus \langle -2(n+1) \rangle.$ Thus, we have
$$
L^{\vee}/L \simeq \mathbb{Z}/{2(n+1)}\mathbb{Z}.
$$
\end{Example}

Recall that the lattice $(L,q)$ of the second cohomology of IHSM with Beauville--Bogomolov form is the indefinite lattice of signature $(3,b_2-3),$ where $b_2$ is the second Betti number.

Denote by $O(L^{\vee}/L)$  the subgroup of $\mathrm{Aut}(L^{\vee}/L)$ consisting of multiplication by all elements of $t \in L^{\vee}/L$ (viewed as $\mathbb{Z}/2k\mathbb{Z}$ for natural number $k$) such that~\mbox{$t^2=1.$} Let~$\mathrm{Gr}^+(L)$ be the Grassmannian parametrizing maximal positive definite subspaces in~$L_{\mathbb{R}}$ and let $\mathrm{Gr}^{+,\mathrm{or}}(L)$ be the oriented Grassmannian parametrizing oriented maximal positive definite subspaces in $L_{\mathbb{R}}.$ Note that $\mathrm{Gr}^{+,\mathrm{or}}(L)$ is a double cover of $\mathrm{Gr}^+(L).$ For the Grassmannian  $\mathrm{Gr}^+(L)$ we have the following proposition.

\begin{Prop}\label{Prop:grass}
Let $(L,q)$ be an indefinite lattice. Then~$\mathrm{Gr}^+(L),$ parametrizing   maximal positive definite subspaces in $L_{\mathbb{R}},$ is contractible.
\end{Prop}

\begin{proof}
Let $L_{\mathbb{R}}=V_+ \oplus V_{-}$ be a decomposition of the vector space $L_{\mathbb{R}}$ in the sum of vector subspaces such that~\mbox{$q\vert_{V_+}$} is positive definite and $q\vert_{V_-}$ is negative definite. Assume that~\mbox{$V_+ \perp V_-$} with respect to $q.$ Let $U \subset L_{\mathbb{R}}$ be a maximal positive subspace. Then~$U$ can be seen as a graph of some linear map $V_+ \to V_-.$ Indeed, the map $f:U \to V_+ \times V_-$ can be defined as 
$$
U \ni u \mapsto (\mathrm{pr}_+(u), \mathrm{pr}_-(u)),
$$

\noindent where $\mathrm{pr}_+$ is a projection to the subspace $V_+$ and $\mathrm{pr}_-$ is a projection to the subspace~$V_-.$ Obviously, the image is isomorphic to $V_+,$ because the kernel of the composition $\mathrm{pr}_+ \circ f$ is zero, since $U \cap V_-=0$ and $\dim U= \dim V_+.$ 

So let $U$ be the graph of the linear map $\phi.$ Then $U$ is positive definite if and only if for any non-zero $v \in V_+$ we get 
\begin{equation}\label{eq:grassmannian}
q(v+\phi(v))>0.
\end{equation}
\noindent As \eqref{eq:grassmannian} holds for any $t\phi,$ where $t \in [0,1],$  we obtain the contraction of $\mathrm{Gr}^+(L).$

\end{proof}

Therefore, since $L$ is indefinite, by Proposition \ref{Prop:grass}  we get that $\mathrm{Gr}^+(L)$ is contractible and, thus, the Grassmannian $\mathrm{Gr}^{+,\mathrm{or}}(L)$ has two connected components. Denote by~$O^+(L)$ the subgroup of $O(L)$ acting trivially on $\pi_0\left(\mathrm{Gr}^{+,\mathrm{or}}(L)\right)$ and by~$SO^+(L)$ the subgroup of $O^+(L)$ which consists of the elements with determinant equal to $1.$ Let 
$$
\widetilde{O(L)} \,=\,\{g \in O(L) \mid  g(l^*) \equiv l^*  \mod L \quad \text{for all} \; l^* \in L^{\vee} \}
$$
\noindent and let $\widetilde{SO^+(L)}$ be the subgroup of $\widetilde{O(L)}$ which consists of the element with determinant equal to $1$ and acting trivially on $\pi_0\left(\mathrm{Gr}^{+,\mathrm{or}}(L)\right).$

Now let us recall a basic theorem about lattices, which we shall use next.

\begin{Th}[{\cite[Lemma 3.5]{Gritsenko} and \cite[Lemma 2.6]{MonRapag}}]\label{Gritsenko}
Let $L$ be a lattice containing two orthogonal isotropic planes.  Then  the $\widetilde{SO^+(L)}$-orbit of a primitive vector $l \in L$ is uniquely determined  by two invariants: $q(l)$ and~$\delta(l).$

\end{Th}

\begin{Lemma}\label{GritsenkoK3A}
Let $L=H^2(X,\mathbb{Z})$ be a lattice of second cohomologies with Beauville--Bogomolov form, where $X$ is a $\mathrm{K3}$ or Kummer type manifold. Then the monodromy orbit of a primitive vector $l \in L$ is uniquely determined  by two invariants: $q(l)$ and~$\pm\delta(l).$
\end{Lemma}

\begin{proof}
First of all, let us consider $\mathrm{K3}$ type case. Recall that by Example~\ref{ex:K3lattice} the lattice of the second cohomology is
$$
L=U^{\oplus 3} \oplus E(-8)^{\oplus 2} \oplus \langle -2(n-1) \rangle.
$$
\noindent Consider 
$$
\pi: O^+(L) \to O(L^{\vee}/L),
$$
\noindent  Then, according to \cite[Lemma 9.2]{Markman}, the group $\mathrm{Mon}(X)$ is equal to $\pi^{-1}(\{1,-1\}).$ Let $m,l \in L$ lie in the same monodromy orbit. This means that there is~\mbox{$\gamma \in \mathrm{Mon}(X)$} such that~\mbox{$m=\gamma(l).$}  Therefore, by the definition of $\gamma$ we get that $q(l)=q(\gamma(l))$ and~\mbox{$\delta(l)=\pm\delta(\gamma(l)).$}

In the opposite direction, assume that there are primitive elements $l$ and $m$ which satisfy
$$
q(m)=q(l) \quad \text{and} \quad \delta(m)=\pm\delta(l).
$$
\noindent Then if $\delta(m)=\delta(l),$ by Theorem~\ref{Gritsenko} there is $\gamma \in \widetilde{SO^+(L)}$ such that $m=\gamma(l)$ and we are done. Otherwise, if $\delta(m)=-\delta(l),$ then $\delta(m)=\delta\left(\mathrm{R}_e(l)\right)$ and by  Theorem~\ref{Gritsenko} there is $\gamma \in \widetilde{SO^+(L)}$ such that $m=\gamma(\mathrm{R}_e(l))$ and we are done.

 Now let us consider the Kummer case. For Kummer variety the lattice is equal to $L = U^{\oplus 3} \oplus \langle -t \rangle.$ Let $\pi: O^+(L) \to O(L^{\vee}/L).$  According to~\mbox{\cite[Theorem 2.3]{Mongardi}} the monodromy group consists of two types of elements, namely,  
$$
\pi^{-1}(1) \cap SO^+(L) \quad \text{and} \quad \pi^{-1}(-1) \cap (O^+(L) \setminus SO^+(L)).
$$

\noindent  Let $m,l \in L$ lie in the same monodromy orbit. This means that there is~\mbox{$\gamma \in \mathrm{Mon}(X)$} such that~\mbox{$m=\gamma(l).$}  Again by the definition of $\gamma,$ it is obvious that 
$$
q(l)=q(\gamma(l)) \;\; \text{and} \;\; \delta(l)=\pm\delta(\gamma(l)).
$$

In the opposite direction, assume that there are primitive elements $l$ and $m$ which satisfy
$$
q(m)=q(l) \quad \text{and} \quad \delta(m)=\pm\delta(l).
$$
\noindent Then if $\delta(m)=\delta(l),$ by Theorem~\ref{Gritsenko} there is $\gamma \in \widetilde{SO^+(L)}$ such that $m=\gamma(l)$ and we are done. Otherwise, if $\delta(m)=-\delta(l),$ then $\delta(m)=\delta\left(\mathrm{R}_e(l)\right)$ and by  Theorem~\ref{Gritsenko} there is $\gamma \in \widetilde{SO^+(L)}$ such that $m=\gamma(\mathrm{R}_e(l))$ and we are done, because 
$$
\gamma\mathrm{R}_e \in O^+(L) \setminus SO^+(L) \quad \text{and} \quad \det\left(\gamma(\mathrm{R}_e(l))\right)=1.
$$

\end{proof}

\section{Description of nef cone for $\mathrm{K3}$ type manifolds}
In this section we prove Theorem \ref{TheoremK3}. Let $X$ is a $\mathrm{K3}$ type manifold of dimension $2n.$ Denote by~\mbox{$e \in H^2(X,\mathbb{Z})$} the class of half of the exceptional divisor on a deformation of~$X$ which is the Hilbert scheme of $n$ points on some $\mathrm{K3}$ surface. We identify~$H^2(X,\mathbb{Z})$ and $H^2(Y,\mathbb{Z})$ for any deformation $Y$ of $X$ in $\mathrm{Teich},$ since topologically they are the same. Half of the exceptional divisors on a deformation of $X$ which is a Hilbert scheme of points on a $\mathrm{K3}$ surface is then $\gamma(e)$ for some~\mbox{$\gamma \in \mathrm{Mon}(X).$}    By~\mbox{$\hat{e}\,=\,\frac{e}{2(n-1)} \in H_2(X,\mathbb{Z})$} we denote the dual class of $e.$  First of all, we need the following lemma (cf.~\cite[Lemma 2.1 and Remark 2.2]{Ciliberto}).

\begin{Lemma}\label{RHK3}
Let $S$ be a $\mathrm{K3}$ surface. Let $C \subset S$ be a smooth curve of genus $g.$  Assume that there is a linear system $g^1_k$ on $C$ with $1 \leqslant  k \leqslant  n$ for some $n \in \mathbb{N}.$   Then there is a map $f \colon C \to \mathrm{Hilb}^n(S)$ such that its image  is a rational curve $R$ which is numerically equivalent to $C-(g-1+k)\hat{e}$ in $\mathrm{N}_1(\mathrm{Hilb}^n(S),\mathbb{Z}).$

\end{Lemma}
\begin{proof}
The lemma follows from the Riemann-Hurwitz formula.  The map $f$ is induced by a linear system $g^1_k$ on $C$ and is given by
$$
C \stackrel{g^1_k}{\longrightarrow} \mathbb{P}^1 \subset  \mathrm{Sym}^k(C) \hookrightarrow \mathrm{Sym}^k(S) \hookrightarrow \mathrm{Sym}^n(S),
$$

\noindent where the last inclusion $\mathrm{Sym}^k(S) \hookrightarrow \mathrm{Sym}^n(S)$ is given by 
$$
\mathrm{Sym}^k(S) \ni x_1+\ldots +x_k \mapsto x_1+\ldots+x_k+y_{k+1}+\ldots+y_n \in \mathrm{Sym}^n(S)
$$
\noindent for $y_{k+1}, \ldots, y_n \in S$ which are distinct fixed points out of $C.$  The ramification points are exactly the intersection of the rational curve $f(C)$ with the exceptional divisor. Let us write $R \equiv C-m\hat{e}$ with  
$$
C \in N_1(S) \hookrightarrow N_1(\mathrm{Hilb}^n(S))
$$
\noindent  and $m \geqslant 0.$ We have $q(e,\hat{e})\,=\,-1,$ thus, we get 
$$
m\,=\,q(R,e)\,=\,\frac{1}{2}q(R,\Delta),
$$
\noindent  where $\Delta\,=\,2e$ is a class of the diagonal in $\mathrm{Sym}^n(S).$   At the same time, we get 
$$
\frac{1}{2}q(R,\Delta)\,=\,\frac{1}{2}2(g-1+k)\,=\,(g-1+k),
$$
\noindent where the first equality is true by Riemann-Hurwitz formula.

\end{proof}

We have the following important property which will be very useful for proving the minimality of curve in the process of proving Theorem~\ref{TheoremK3}.

\begin{Lemma}\label{lemma:K3k>2g-2}
Let $M=\mathrm{Hilb}^n(S)$ for a $\mathrm{K3}$ surface $S$ such that~\mbox{$\mathrm{Pic}(S)=\langle x \rangle$} and let~\mbox{$C \subset S$} be a smooth curve of genus $g$ which corresponds to the class $x \in H^2(S, \mathbb{Z}).$ Let $t=2(n-1).$ Let $e \in  H^{1,1}(M,\mathbb{Z})$ be a class of half of the exceptional divisor on~$M.$ Assume that $\beta=rx-be \in H^{1,1}(M,\mathbb{Z})$ is a primitive class such that $r \in \mathbb{N},$
\begin{equation}\label{eq:K3k>2g-2b}
b \in [0,\frac{r}{2}] \quad \text{is an integer number,}
\end{equation}
\noindent and 
\begin{equation}\label{eq:K3k>2g-2bq>0}
q(\beta)<0.
\end{equation}
\noindent Then $k=\frac{bt}{r}-g+1>2g-2.$

\end{Lemma}

\begin{proof}
If $g=0$ or $1,$ we are done since $\frac{br}{t}>0.$ So assume that $g \geqslant 2.$ By \eqref{eq:K3k>2g-2bq>0} we obtain
\begin{equation}\label{eq:K3frac<b/r}
\sqrt{\frac{2g-2}{t}} < \frac{b}{r}.
\end{equation}
\noindent So by \eqref{eq:K3k>2g-2b} we get
\begin{equation}\label{eq:K3sqrt>2}
\sqrt{\frac{t}{2g-2}} > 2. 
\end{equation}
\noindent By \eqref{eq:K3frac<b/r} and \eqref{eq:K3sqrt>2} we have the following inequalities
$$
k=\frac{bt}{r}-g+1>\sqrt{\frac{2g-2}{t}}-g+1=(2g-2)\left(\sqrt{\frac{t}{2g-2}} -\frac{1}{2}  \right)>2g-2.
$$

\end{proof}

\begin{Lemma}\label{lemma:K3minimalcurves}
Let $M=\mathrm{Hilb}^n(S)$ for a $\mathrm{K3}$ surface $S$ such that~\mbox{$\mathrm{Pic}(S)=\langle x \rangle$} and let~\mbox{$C \subset S$} be a smooth curve of genus $g$ which corresponds to the class $x \in H^2(S, \mathbb{Z}).$ Let $t=2(n-1).$ Let $e \in  H^{1,1}(M,\mathbb{Z})$ be a class of half of the exceptional divisor on~$M.$ Assume that $\beta=rx-be \in H^{1,1}(M,\mathbb{Z})$ is a primitive class such that $r$ divides $t,$ $b \in [0,\frac{r}{2}]$ is an integer number, and $q(\beta)<0.$ Then there is a rational curve $R \subset M$ corresponds to the class $\beta.$ Moreover, the curve $R$ varies in the variety~\mbox{$\mathcal{Z} \subset M$} which is a rational quotient such that dimension of  its fibres  is  equal to codimension of $\mathcal{Z}$ in $M.$ 
\end{Lemma}

\begin{proof}
Let us consider $\widehat{\beta}=x-\frac{b}{r}e \in H_2(X,\mathbb{Q}).$ Let $m=\frac{t}{r}$ We have 
$$
\widehat{\beta} = x-\frac{mb}{t}e \equiv C-mb\hat{e}.
$$

\noindent Let us prove that there is a linear system $g^1_k$ on $C$ such that 
\begin{equation}\label{eq:equationofkK3'}
k=mb-g+1=\frac{bt}{r}-g+1
\end{equation}
\noindent and $k \leqslant n.$  The inequality is clear, because this means that 
$$
\frac{bt}{r}-g+1 \leqslant  \frac{t}{2}+1=n
$$
\noindent or, equivalently, $\frac{b}{r} \leqslant  \frac{1}{2}+\frac{g}{t}$ which is true according to the conditions on~$b.$

 By Lemma~\ref{lemma:K3k>2g-2} we get 
\begin{equation}\label{eq:k>2g-2minimalityK3}
k>2g-2.
\end{equation}
\noindent If we prove that  the gonality of $C$ is less than or equal to $k,$ then by Lemma~\ref{RHK3} we get that the linear system $g^1_k$ gives a rational curve $R$ on $M.$ To do this we need to use the inequality for the gonality in~\cite[Chapter~2, page 261]{grif}. In other words, we  have to prove that
$$
\floor*{\frac{g+1}{2}}+1 \leqslant  k.
$$
\noindent For $g \geqslant 2$ this immediately follows from Lemma~\ref{lemma:K3k>2g-2}. For $g=0$ it is enough to prove that $k \geqslant 1.$ Form \eqref{eq:equationofkK3'} it is enough to prove that $\frac{bt}{r}+1 \geqslant 1,$ which is true. For~\mbox{$g=1$} it is enough to prove that $k \geqslant 2.$ Form \eqref{eq:equationofkK3'} it is enough to prove that~\mbox{$\frac{bt}{r} \geqslant 1,$} which is true. This means that  there is a rational curve $R$ which corresponds to the class~$\beta.$ 

Now let us find the MBM locus $\mathcal{Z}$ of the curve $R.$ Assume that the genus of the curve~$C$ is $g>0.$ By \eqref{eq:k>2g-2minimalityK3} we get that~$\mathrm{Sym}^k(C)$ is a projective bundle over the Jacobian~$J(C)$ of $C.$ Also by Hirzebruch-Riemann-Roch theorem we get that 
$$
\dim H^0(S,\mathcal{O}_S(C))=g+1
$$
\noindent and so $\mathbb{P}^g$ parametrizes curves of genus $g$ on the surface $S.$ Thus, we have the fibration over $\mathbb{P}^g$ with the Jacobians as the fibres. In other words, we get the fibration 
\begin{equation}\label{eq:beauville-mukaisystemK3}
j \colon \mathcal{J} \to \mathbb{P}^g,
\end{equation}
\noindent such that the fibre of $j$ over a point $D \in \mathbb{P}^g,$ where $D \subset S$ is a curve of genus $g,$ is equal to the Jacobian of $D.$ Let 
$$
\phi \colon \mathcal{P} \to \mathcal{J}
$$
\noindent  be a fibration such that the fibre over a point $J \in \mathcal{J}$ is $\mathbb{P}^{k-g}$ which gives us the fibration
$$
\mathrm{Sym}^k(D) \xrightarrow{\mathbb{P}^{k-g}} J(D)
$$
\noindent for a curve $D \in \mathbb{P}H^0(S,\mathcal{O}_S(C))^{\vee}$ of genus $g.$  The completion of $\mathcal{Z}$ is $\mathcal{P} \times \mathrm{Sym}^{n-k}(S).$ Note that the fibration \eqref{eq:beauville-mukaisystemK3} is called the Beauville-Mukai system and it is known that $\mathcal{J}$ is deformation equivalent to $\mathrm{Hilb}^g(S)$ (see~\cite[Section 2.3]{Sawon}). In other words, the desingularization of $\mathcal{Z}$ is deformation equivalent to $\mathbb{P}^{k-g}$-bundle over IHSM. The dimension of $\mathcal{Z}$ is $2n-2k+g+g+k-g=2n-k+g.$ The dimension of the fibre of its rational quotient is equal to $k-g$ which is the codimension of $\mathcal{Z}$ in $M.$

If the genus of $C$ is $g=0,$ then the MBM locus $\mathcal{Z}$ of $R$ is 
$$
\mathrm{Sym}^k(C) \times \mathrm{Sym}^{n-k}(S) \simeq \mathbb{P}^k \times \mathrm{Sym}^{n-k}(S).
$$
\noindent Its dimension is equal to $2n-k$ and the dimension of the fibre of its rational quotient is equal to the codimension of $\mathcal{Z}$ in $M.$ This completes the proof.

\end{proof}

Now we are ready to prove Theorem \ref{TheoremK3}. 

\begin{proof}[Proof of Theorem \ref{TheoremK3}.]

Let $\alpha$ be a primitive element in $H^2(X,\mathbb{Z})$ such that~\mbox{$q(\alpha)<0.$} First of all, assume that the class $\alpha \in H^2(X,\mathbb{Z})$  has the same square and the image in the discriminant group as $e$, then according to Lemma~\ref{GritsenkoK3A}  they are in the same monodromy orbit. The class $e$  is clearly an MBM class, because it corresponds to the general fibre $R$ of the exceptional divisor over the diagonal in~$\mathrm{Sym}^n(S)$ which is MBM locus of $R.$ This MBM locus is $\mathbb{P}^1$-bundle.  Therefore, the class $\alpha$ is also an MBM class. Thus, by Theorem \ref{Torelli} there is a deformation of~$X$ preserving~$e,$ which is  isomorphic to $\mathrm{Hilb}^n(S)$ with 
$$
\mathrm{Pic}(\mathrm{Hilb}^n(S))\,=\,\langle  e \rangle,
$$
\noindent i.e. $S$ does not have curves. We have
$$
q(\widehat{e})=-\frac{1}{2(n-1)}=q(\widehat{\alpha}).
$$
\noindent  Now assume that $\alpha$ and $e$ are not in the same monodromy orbit. We divide the proof into several steps.

\textbf{Step 1.} \emph{In this step we recall that there is a deformation $Y$ of $X$ within~$\textup{Teich}_{\alpha}$ with Picard lattice $\mathrm{Pic}(Y)=\langle \alpha, e \rangle$  such that $Y$ is isomorphic to the Hilbert scheme of $\mathrm{K3}$ surface.}  This immediately follows from  Corollary~\ref{Corollary:deformation}. So $Y=\mathrm{Hilb}^n(S')$ for a $\mathrm{K3}$ surface $S'$ with 
$$
\mathrm{Pic}(S')=\langle y \rangle=e^{\perp} \; \text{in} \; \langle \alpha,e \rangle.
$$

\textbf{Step 2.}  \emph{In this step we deform $Y$ to $Z$ in $\textup{Teich}_{\alpha}$ such that $\mathrm{Pic}(Z)=\langle \alpha, \gamma^{-1}(e) \rangle$ and $\gamma \in \mathrm{Mon}(X)$ is chosen to normalize $\gamma(\alpha)$  in a certain way.}  Let us prove the following lemma.

\begin{Lemma}\label{lemmaAVK3}
Let $Y$ be a Hilbert scheme of $n$ points $\mathrm{Hilb}^n(S')$ on a $\mathrm{K3}$ surface~$S'.$ Assume that $\mathrm{Pic}(Y)=\langle \alpha, e\rangle,$  where~\mbox{$\alpha=ay+ce \in H^{1,1}(Y,\mathbb{Z})$} is a primitive class with $q(\alpha)<0$   for~\mbox{$y \in \mathrm{Pic}(S').$}  Then there is an element~$\gamma \in \textup{Mon}(Y),$ such that~\mbox{$\gamma(\alpha)=rx-be$}  with primitive $x \in H^2(S',\mathbb{Z}),$ and with integer numbers~\mbox{$b \in [0,\frac{r}{2}]$} and~\mbox{$r=\gcd(a,t),$} where $t=2(n-1).$

\end{Lemma}

\begin{proof}
Let us consider the element $\alpha=ay+ce$ and let
$$
r\,=\,\gcd(a,t), \quad a\,=\,rm \quad \text{and} \quad t\,=\,rs.
$$

\noindent Note that as the lattice of the second cohomologies of a $\mathrm{K3}$ surface is unimodular,  every even number can be realized as a square of some primitive element in the lattice, so we take some primitive 
$$
y' \in H^2(S', \mathbb{Z}) \subset H^2(Y, \mathbb{Z})
$$
\noindent with $q(y')\,=\,m^2q(y).$    Let us first prove that the class $\alpha\,=\,ay+ce$ and $\widetilde{\alpha}\,=\,ry'+ce$ have the same squares and images of $\delta$ (see Definition \ref{definition:divisib}). The equality 
$$
q(\alpha)\,=\,q(\widetilde{\alpha})
$$
\noindent is obvious, because
\begin{gather*}
q(\alpha)\,=\,a^2q(y)-c^2t \;\; \text{and} \\
 q(\widetilde{\alpha})\,=\,r^2q(y')-c^2t\,=\,(rm)^2q(y)-c^2t\,=\,a^2q(y)-c^2t.
\end{gather*}

\noindent The discriminant group is equal to $\mathbb{Z}/t\mathbb{Z}.$ This means that 
$$
\delta(\alpha)\,=\,\frac{\alpha}{r} \mod L,
$$
\noindent because $d(\alpha)\,=\,\gcd(a,t)\,=\,r$ and thus, 
$$
\delta(\alpha)\,=\,cs \mod \, t.
$$ 
\noindent For $\widetilde{\alpha}$ we get 
$$
\delta(\widetilde{\alpha})\,=\,\frac{\widetilde{\alpha}}{r} \mod L,
$$
\noindent  because $d(\widetilde{\alpha})\,=\,\gcd(r,t)\,=\,r$ and thus, 
$$
\delta(\widetilde{\alpha})\,=\,cs \mod \, t.
$$
\noindent Therefore, according to Lemma \ref{GritsenkoK3A} we obtain $\widetilde{\gamma} \in \text{Mon}(Y)$ such that $\widetilde{\gamma}(\alpha)\,=\,\widetilde{\alpha}.$

Now we prove that there is  $\gamma \in \text{Mon}(Y)$  such that $\gamma(\alpha)=rx-be$  with $ 0 \leqslant b \leqslant \frac{t}{2}$ and a primitive $x \in H^2(S',\mathbb{Z}).$ First of all, we find $x \in H^2(S',\mathbb{Z})$ such that
$$
\widetilde{\alpha}\,=\,ry'+ce \quad \text{and} \quad \doublewidetilde{\alpha}\,=\,rx+(c-lr)e
$$ 
\noindent for all $l \in \mathbb{Z}$ are in the same monodromy orbit. The classes $\widetilde{\alpha}$ and $\doublewidetilde{\alpha}$ have the same image in the discriminant group, because 
$$
\delta(\widetilde{\alpha})\,=\,cs \mod \, t \;\; \text{and} \;\; \delta(\doublewidetilde{\alpha})\,=\,(c-lr)s \mod \, t.
$$
\noindent So it is enough to choose $x \in H^2(S',\mathbb{Z})$ such that  the classes $\widetilde{\alpha}$ and $\doublewidetilde{\alpha}$ have the same squares. We get
$$
q(\widetilde{\alpha})\,=\,r^2q(y')-c^2t \quad \text{and} \quad q(\doublewidetilde{\alpha})\,=\,r^2q(x)-c^2t-l^2r^2t+2clrt.
$$

\noindent Thus, $q(\widetilde{\alpha})\,=\,q(\doublewidetilde{\alpha})$ if and only if
\begin{equation}\label{shift}
q(x)\,=\,q(y')+tl^2-2cls.
\end{equation}

We can choose $b'=c-lr$ in   the segment $[-\frac{r}{2},\frac{r}{2}].$ Also we can choose primitive~\mbox{$x \in H^2(S',\mathbb{Z})$} satisfying \eqref{shift}. Let $b=|b'|.$  By Lemma \ref{GritsenkoK3A} there is~\mbox{$\gamma \in \text{Mon}(Y)$} such that $\gamma(\alpha)=rx-be.$ So we are done.

\end{proof}

So consider the point in $\mathrm{Teich}_{\alpha}$ which corresponds to the variety $Z$ such that 
$$
\mathrm{Pic}(Z)=\langle \alpha, \gamma^{-1}(e) \rangle,
$$
\noindent  where $\gamma \in \mathrm{Mon}(Y)$ is as in Lemma~\ref{lemmaAVK3}.  The variety $Z$ is bimeromorphic to the Hilbert scheme of points on some $\mathrm{K3}$ surface $S$ with $\mathrm{Pic}(S)=\langle x \rangle,$ where $x \in \mathrm{Pic}(S')$ as in Lemma~\ref{lemmaAVK3}. Note that in any monodromy orbit there is only one element which satisfies the conditions of Lemma \ref{lemmaAVK3}.

\textbf{Step 3.} \emph{In this step we give necessary condition for $\alpha$ to be MBM.}

\begin{Lemma}\label{lemma:K3thenq>-2}
Let $Y$ be an IHSM of $\mathrm{K3}$ type of dimension $2n.$ Let $\alpha \in H^{1,1}(Y,\mathbb{Z})$ be a primitive class which is MBM. Assume that $Y$ is deformation equivalent to the variety $Z$ such that $\alpha \in H^{1,1}(Z,\mathbb{Z})$ and there is $\gamma \in \mathrm{Mon}(Y)$ such that $\gamma(Z)$ is bimeromorphic to~\mbox{$\mathrm{Hilb}^n(S)$} for some $\mathrm{K3}$ surface $S.$ Assume that $\gamma(\alpha)=rx-be,$ where $\mathrm{Pic}(S)=\langle x \rangle,$  $b \in [0,\frac{r}{2}]$ is an integer number and $r$ divides $t=2(n-1).$ Then~$q(x) \geqslant -2.$

\end{Lemma}

\begin{proof}
 If~\mbox{$q(x)<-2,$} then the  $\mathrm{K3}$ surface $S$ does not have any curves. Therefore, by Remark~\ref{remark:vanishingofcurve} the class $\alpha$ is not an MBM class, because the only rational curves on the Hilbert scheme are of class $\hat{e}.$ If $q(x) \geqslant -2,$ we use the uniform representation of monodromy orbit from Step~2, namely, for every $\alpha \in H^2(X,\mathbb{Z})$ there is $\gamma \in \text{Mon}(X)$ such that 
$$
\gamma(\alpha)=rx-be,
$$
\noindent where $r$ divides~$t,$ the integer number $b \in [0,\frac{r}{2}]$ and $\gcd(r,b)=1.$  Thus, we can deform~$X$ to $Z$ such that  $\gamma(Z)$ is bimeromorphic to $\mathrm{Hilb}^n(S)$ with  
$$
\mathrm{Pic}(S)=\langle x \rangle.
$$

\end{proof}

\textbf{Step 4.} \emph{In this step we prove that the necessary condition on $\alpha$ to be MBM given in Step 3 is also sufficient.}

\begin{Lemma}\label{lemma:K3q>-2}
Let $Y$ be an IHSM of $\mathrm{K3}$ type of dimension $2n.$ Let $\alpha \in H^{1,1}(Y,\mathbb{Z})$ be a primitive class such that $q(\alpha)<0.$ Assume that~$Y$ is deformation equivalent to the variety $Z$ such that $\alpha \in H^{1,1}(Z,\mathbb{Z})$ and there is~\mbox{$\gamma \in \mathrm{Mon}(Y)$} such that $\gamma(Z)$ is bimeromorphic to~$\mathrm{Hilb}^n(S)$ for some $\mathrm{K3}$ surface~$S.$ Assume that $\gamma(\alpha)=rx-be,$ where~\mbox{$\mathrm{Pic}(S)=\langle x \rangle,$ $b \in [0,\frac{r}{2}]$} is an integer number and $r$ divides $t=2(n-1).$ If~\mbox{$q(x) \geqslant -2,$} then the class of $\widehat{\gamma(\alpha)} \in H_2(Z,\mathbb{Q})$ on $Z$ is represented by a smooth rational curve.

\end{Lemma}

\begin{proof}
Consider the dual class $\widehat{\gamma(\alpha)} \in H_2(Z,\mathbb{Q})$ of $\gamma(\alpha) \in H^2(Z,\mathbb{Z}).$ Note that 
$$
\widehat{\gamma(\alpha)}=\frac{\gamma(\alpha)}{d(\gamma(\alpha))}=x-\frac{b}{r}e=x-\frac{bt}{r}\widehat{e}.
$$
\noindent Assume that $ \gamma(Z)$ is bimeromorphic to $\mathrm{Hilb}^n(S)$ with $\mathrm{Pic}(S)=\langle x \rangle$ and $q(x) \geqslant -2$ such that
\begin{equation}\label{eq:b<r/2}
\gamma(\alpha)=rx-be, \;\; r \;\; \text{divides} \;\; t, \;\;  b \in [0,\frac{r}{2}] \;\; \text{is integer and} \;\; \gcd(r,b)=1.
\end{equation}

\noindent So by Lemma~\ref{lemma:K3minimalcurves} we get the desired, i.e. we obtain a smooth rational curve on~$Z.$

\end{proof}

\begin{Cor}\label{cor:q>2thusmbmK3}
Let $Y$ be an IHSM of $\mathrm{K3}$ type of dimension $2n.$ Let $\alpha \in H^{1,1}(Y,\mathbb{Z})$ be a primitive class. Assume that~$Y$ is deformation equivalent to the variety $Z$ such that $\alpha \in H^{1,1}(Z,\mathbb{Z})$  and there is~\mbox{$\gamma \in \mathrm{Mon}(Y)$} such that $\gamma(Z)$ is bimeromorphic to~$\mathrm{Hilb}^n(S)$ for some $\mathrm{K3}$ surface~$S$ with $\mathrm{Pic}(S)=\langle x \rangle.$  Assume that $\gamma(\alpha)=rx-be,$ where  $b \in [0,\frac{r}{2}]$ is an integer number and $r$ divides $t.$ If $q(x) \geqslant -2,$ then the class~$\alpha$ is MBM.

\end{Cor}

\begin{proof}
By Lemma~\ref{lemma:K3q>-2} we get that the class of $\widehat{\gamma(\alpha)} \in H_2(Z,\mathbb{Q})$ on $Z$ is represented by a smooth rational curve. Since it is of negative square by definition of MBM curve it is enough to prove that the curve $R$ which represents $\widehat{\gamma(\alpha)}$ is minimal. The minimality follows from Remark~\ref{remark:criteriaminimality} and  Lemma~\ref{lemma:K3minimalcurves}, since by Theorem~\ref{th:analiticallyequiv} it is enough to find MBM locus of $R$ on $\mathrm{Hilb}^n(S).$

\end{proof}

\textbf{Step 5.} \emph{In this step we prove that $\alpha \in H^2(X,\mathbb{Z})$ is MBM  if and only if its rational multiple satisfies the assumption of Theorem~\ref{TheoremK3}.}  Recall that an element~$\alpha$  is MBM if and only if the manifold $X$ can be deformed so that on the deformed variety the class $\widehat{\alpha} \in H_2(X, \mathbb{Q})$ is represented by an extremal rational curve. By Lemma~\ref{lemmaAVK3} the element $\alpha$ lies in the same monodromy orbit as the element 
$$
\alpha'=rx-b'e,
$$
\noindent where $r$ divides $t,$  $b' \in [0,r]$ and $\gcd(r,b)=1.$ Let $t=rs,$ where $s \in \mathbb{N}.$ Note that by Lemma~\ref{GritsenkoK3A} we have $q(\alpha)=q(\alpha')$ and~\mbox{$\delta(\alpha)=\pm \delta(\alpha').$} By direct calculation we get that $\delta(\alpha')=-b's.$ Therefore, we obtain
$$
-\delta(\alpha) = -b's \in [0,n-1].
$$
\noindent As 
$$
q(\widehat{\alpha})=q(x-\frac{b'}{r}e)=2a-\frac{b'^2t}{r^2}=2a-\frac{b'^2s^2}{t}
$$
\noindent for $q(x)=2a,$ where $q(x)$ is even because the lattice of the second cohomologies of $\mathrm{K3}$ surface is even. By Step $3$ and $4$ we get that $\alpha$ is MBM if and only if $q(x) \geqslant -2.$ This means that
$$
-2 \leqslant 2a <\frac{b^2}{2(n-1)}.
$$

\textbf{Step 6.} \emph{In this step we prove that the numerical conditions on MBM class~$\alpha$ defined by~\eqref{eq:bin0,n-1K3},~\eqref{eq:q(alpha)K3} and~\eqref{eq:-2<2a<bK3} uniquely determine its monodromy orbit.}  Some calculations will be similar to the calculations which we did in the proof of Lemma~\ref{lemmaAVK3}. Let us fix integer numbers $a$ and $b$ which satisfy~\eqref{eq:bin0,n-1K3},~\eqref{eq:q(alpha)K3} and~\eqref{eq:-2<2a<bK3}.  Let us find all primitive $\alpha \in H^2(X,\mathbb{Z})$ such that
\begin{equation}\label{eq:deltaq(hata)K3}
\delta(\alpha)=\pm b \quad \text{and} \quad q(\widehat{\alpha})=2a-\frac{b^2}{2(n-1)}.
\end{equation}

 Consider $\alpha=fx+de$ on some deformation $Y$ of $X,$ such that $Y$ is bimeromorphic to the Hilbert scheme of points on a $\mathrm{K3}$ surface $S,$ where $f,d \in \mathbb{Z},$ $e \in H^2(X, \mathbb{Z})$ is the class of half of the exceptional divisor on $Y$ and $x \in H^2(S,\mathbb{Z}).$   Recall (by Definition~\ref{definition:divisib}) that 
$$
d(\alpha)=\gcd(f,2(n-1)).
$$
\noindent Let $f=f' d(\alpha).$  Since $\delta(\alpha)=\pm b,$ by the definition of $\delta(\alpha)$ we get that 
$$
d=\pm \frac{bd(\alpha)}{2(n-1)} + cd(\alpha)
$$
\noindent for some integer $c.$ Note that by our assumptions $d$ is also integer. We have
\begin{equation}\label{eq:alpa=b+cdK3}
\alpha=f'd(\alpha)x+\left(\pm \frac{bd(\alpha)}{2(n-1)} + cd(\alpha)\right)e.
\end{equation}
\noindent  So by \eqref{eq:deltaq(hata)K3}  a primitive $\alpha \in H^2(X,\mathbb{Z})$ satisfies \eqref{eq:deltaq(hata)K3} if and only if it is of the form~\eqref{eq:alpa=b+cdK3} such that 
$$
f'^2q(x)-2(n-1)c^2 \mp 2bc=2a.
$$ 
\noindent By direct computations it is not hard to see that any $\alpha$ of the form \eqref{eq:alpa=b+cdK3} lies in the same monodromy orbit as
$$
\alpha'=d(\alpha)y-\frac{bd(\alpha)}{2(n-1)}e,
$$
\noindent where $y \in H^2(X,\mathbb{Z})$ such that $q(y,e)=0$ and $q(y)=2a.$  Note that $\alpha'$ satisfies~\eqref{eq:deltaq(hata)K3}.  If $b=0,$ then $d(\alpha)=1,$ so $\alpha'=z$ such that $q(z)=-2.$ By Lemma~\ref{GritsenkoK3A} all such elements lie in the same monodromy orbit. Let $b \neq 0.$ It is enough to prove that all elements of the form 
$$
\alpha'=mz-\frac{bm}{2(n-1)}e,
$$
\noindent where  
$$
\frac{bm}{2(n-1)} \in \mathbb{N}, \quad \gcd\left(m,\frac{bm}{2(n-1)}\right)=1  \quad \text{and} \quad m \;\; \text{divides} \;\; 2(n-1),
$$
\noindent  which satisfy~\eqref{eq:deltaq(hata)K3}, lie in the same monodromy orbit. Indeed, let
$$
r=\gcd(b,2(n-1)).
$$
\noindent Let $b=b' \cdot r$ and $2(n-1)=v \cdot r.$ Then since 
$$
\frac{bm}{2(n-1)}=\frac{b' \cdot m}{v}
$$
\noindent is integer number, we get that $v$ divides $m.$ By   the condition $ \gcd(m,\frac{bm}{2(n-1)})=1$ we get that 
$$
m=v=\frac{2(n-1)}{r}.
$$

\noindent This means that all primitive $\alpha$ which satisfy \eqref{eq:deltaq(hata)K3} lie in the same monodromy orbit as the class 
$$
\alpha'=\frac{2(n-1)}{\gcd(b,2(n-1))}z-\frac{b}{\gcd(b,2(n-1))}e,
$$
\noindent where $q(z)=2a$  on some deformation $Y$ of $X,$ such that $Y$ is bimeromorphic to the Hilbert scheme of points on a $\mathrm{K3}$ surface $S,$ where  $e \in H^2(X, \mathbb{Z})$ is the class of half of the exceptional divisor on $Y$ and $z \in H^2(S,\mathbb{Z}).$

\end{proof}

\begin{proof}[Proof of Corollary~\ref{cor:TheoremK3}]
Assertion \ref{1} follows from  Lemma~\ref{GritsenkoK3A}. Assertion \ref{2} follows from Lemma~\ref{lemma:K3thenq>-2} and Corollary~\ref{cor:q>2thusmbmK3}. Note that if $(a,b)=(0,1),$ then by   Lemma~\ref{GritsenkoK3A} the elements $e$ and $\alpha$ lie in the same monodromy orbits.

\end{proof}

\begin{Cor}
Let $X$ be an IHSM of $\mathrm{K3}$ type of dimension $2n.$ Let $\alpha \in H^2(X,\mathbb{Z})$ be an MBM class. Assume that $\alpha$ up to a rational multiple can be represented by
$$
2(n-1) \cdot x-be
$$
\noindent with integer number $b \in [0,n-1]$ on some deformation $Y=\mathrm{Hilb}^n(S)$ of $X$ with Picard group $\mathrm{Pic}(Y)=\mathrm{Pic}(S) \oplus \mathbb{Z}e,$  such that $S$ is a $\mathrm{K3}$ surface with the Picard group  $\mathrm{Pic}(S)=\mathbb{Z}x,$  $x^2 \geqslant -2$ and $e$ is a class of half of the exceptional divisor on~$Y.$ Then
$$
g=\frac{x^2}{2}+1 \leqslant  \lceil \frac{n+3}{4} \rceil-1.
$$
\end{Cor}

\begin{proof}
Since $\alpha$ is MBM we get that $q(\alpha)<0.$ So we have
$$
q(2(n-1) \cdot x-be)=4(n-1)^2x^2-2(n-1)b^2<0.
$$
\noindent Therefore, we get $g<\frac{b^2}{4(n-1)}+1 \leqslant \frac{n+3}{4}.$ Thus, we obtain the inequality 
$$
g \leqslant  \lceil \frac{n+3}{4} \rceil - 1 .
$$

\end{proof}

\begin{Cor}[\textbf{of the proof}]\label{iffK3}
Let $X$ be an IHSM of $\mathrm{K3}$ type of dimension~$2n.$ Let $\alpha \in H^2(X,\mathbb{Z})$ be a primitive class and $\widehat{\alpha}\,=\,x-\frac{b}{2(n-1)}e$ be its dual with integer number~\mbox{$b \in [0,n-1]$} on some deformation $Y$ of $X,$ such that~\mbox{$Y \simeq \mathrm{Hilb}^n(S),$} where~$S$ is a $\mathrm{K3}$ surface such that  $\mathrm{Pic}(S)=\mathbb{Z}x.
$ Then $\alpha$ is MBM if and only if
$$
q(x) \geqslant -2 \quad  \text{and} \quad q(\alpha)\,<\,0.
$$
\end{Cor}
\begin{Cor}[{\cite[cf. Theorem 2]{Bakker}}]
Let $X$ be an IHSM of $\mathrm{K3}$ type of dimension~$2n.$ Let $\alpha \in H^2(X,\mathbb{Z})$ be a primitive class and $\widehat{\alpha} \in H^2(X, \mathbb{Q})$ be its dual class. Let $\alpha$ be MBM. Then we get
$$
q(\widehat{\alpha}) \geqslant -\frac{n+3}{2}.
$$
\noindent Moreover, the equality holds if and only if $\alpha$ is a class of the line in the Lagrangian space $\mathrm{Sym}^n(C),$ where $C$ is a smooth rational curve on $\mathrm{K3}$ surface. In this case the MBM locus of the class $\alpha$  on $X$ is bimeromorphic to  $\mathrm{Sym}^n(C).$
\end{Cor}

\begin{proof}

By Theorem~\ref{TheoremK3} we get that either $q(\widehat{\alpha})=-\frac{1}{2(n-1)},$ or $q(\widehat{\alpha})\,=\,2a-\frac{b^2}{2(n-1)},$  where $b \in [0,n-1]$ is an integer number and  $a$ is an integer number which satisfies the inequalities
$$
 -2 \leqslant 2a<\frac{b^2}{2(n-1)}.
$$

\noindent We get
$$
q(\widehat{\alpha})\,=2a-\frac{b^2}{2(n-1)} \geqslant -2-\frac{(n-1)}{2}\,=\,-\frac{n+3}{2}.
$$

\noindent For the second part we refer to~\cite[Example 4.11]{Tschinkel}, as the equality holds if and only if $q(x)\,=\,-2$ and $b\,=\,n-1.$ Actually, this follows from Lemma~\ref{RHK3}. Indeed,  we have
$$
\widehat{\alpha}=C-(n-1)\widehat{e}
$$
\noindent for some rational curve $C$ on a $\mathrm{K3}$ surface. Let $R$ is a rational curve in the class~$\widehat{\alpha}.$ This means that $C$ maps to the rational curve $R$ by the map of degree $n$ and the intersection of $R$  with the general fibre of the exceptional divisor of the Hilbert--Chow morphism is equal to $n-1.$  Linear systems~$g^1_n$ on the rational curve~$C$ are parametrized by~$\mathbb{P}^n$ and this gives us the third part. Therefore, by Theorem~\ref{th:analiticallyequiv} in this extremal case for the $\mathrm{K3}$ type manifolds the MBM locus is bimeromorphic to $\mathbb{P}^n.$

\end{proof}

A similar description of  MBM classes for both $\mathrm{K3}$ and Kummer cases was given in the paper~\cite{Knutsen}. However, the mentioned paper relies on the Brill-Noether theory, whereas our  approach is elementary.

Finally we prove Proposition  \ref{BMeqAV}.

\begin{proof}[Proof of Proposition \ref{BMeqAV}]

According to Bayer-Macr\`{i} theory we get that the walls are the hyperplanes in $\mathrm{Pos}(X)$ orthogonal to $a$ with $a^2 \geqslant -2$ ($a^2$ in a sense of Mukai scalar product). The vector $a\,=\,(u,\varkappa,s)$ belongs to 
$$
H^0(S,\mathbb{Z}) \oplus \mathrm{NS}(S) \oplus H^4(S,\mathbb{Z}).
$$
\noindent for a $\mathrm{K3}$ surface $S.$ To pass to our situation we need to project this vector on the subspace $v^{\perp},$ i.e. to consider~\mbox{$z\,=\,v^2a-(v,a)v.$} Obviously, this vector is orthogonal to $v.$ After some simple calculations we have 
$$
z\,=\,v^2\varkappa-(a,e)e,
$$
\noindent where without loss of generality we put $v\,=\,(1,0,1-n)$ and $e\,=\,(1,0,n-1).$ So the walls in $\mathrm{Pos}(X)$ are hyperplanes orthogonal to $z\,=\,v^2\varkappa-(a,e)e.$

By Theorem \ref{TheoremK3} and its proof $z$ is MBM if and only if we can deform $X$ to $Z$ such that there is $\gamma \in \mathrm{Mon}(X)$ such that $\gamma(Z)$ is bimeromorphic to $\mathrm{Hilb}^n(S)$ with a $\mathrm{K3}$ surface $S.$ Moreover, the following properties hold: the  Picard group is 
$$
\mathrm{Pic}(Z)\comm{_{\mathbb{Q}}}=\langle z, \gamma^{-1}(e) \rangle
$$
\noindent with $\gamma(z)=tx-be,$ where $b \in [0,\frac{t}{2}]$ is an integer,~\mbox{$q(x) \geqslant -2$} such that 
\begin{equation}\label{eq:BMq=}
q(z)=q(\gamma(z))
\end{equation}
\noindent  and $\gamma \in \text{Mon}(Z).$ As $(a,e)=-s-\frac{tu}{2},$ we have $b=\frac{tu}{2}-s,$ because by conditions of Theorem \ref{Theorem12.1} we have 
$$
0 \leqslant  \frac{tu}{2}-s=-s-\frac{tu}{2}+tu \leqslant  \frac{t}{2}.
$$
\noindent From \eqref{eq:BMq=} we get
$$
t^2q(\varkappa)-\left( \frac{tu^2}{4}+s^2+tus \right) t\,=\,t^2q(x)-\left( \frac{tu^2}{4}+s^2-tus \right) t.
$$

\noindent So we obtain $q(x)=\varkappa^2-2us=a^2 \geqslant -2.$ Thus, $z$ is an MBM class.

In the opposite direction, assume that $\alpha$ is an MBM class. Thus, there is a deformation $Z$ of $X$ such that  $\gamma(Z)$ is bimeromorphic to $\mathrm{Hilb}^n(S)$ for some $\mathrm{K3}$ surface $S,$ where 
$$
s\gamma(\alpha)=tx-be
$$
\noindent  for some positive integer $s$ which divides $t.$ We have $b \in [0,\frac{t}{2}].$ Thus, let us take 
$$
a=\left(u,\varkappa,\frac{tu}{2}-b\right)
$$
\noindent  for some integer number $u$ and $\varkappa^2=x^2+2u\left(\frac{tu}{2}-b\right).$ So we get that $a$ satisfies the conditions of Theorem \ref{Theorem12.1}. Indeed, 
$$
a^2=\varkappa^2-2u\left(\frac{tu}{2}-b\right)=x^2 \geqslant -2
$$
\noindent and
$$
(a,v)=\frac{tu}{2}-\frac{tu}{2}+b=b
$$
\noindent Thus, we get $0 \leqslant  (a,v) \leqslant  \frac{t}{2}.$ Moreover, we have that 
$$
z=t\varkappa-(a,e)e=t\varkappa-be
$$
\noindent is a projection of $a$ on $v^{\perp}.$

\end{proof}

\section{Description of nef cone for Kummer type manifolds}

In this section we prove Theorem \ref{TheoremKummer}. Let $X$ be a Kummer type manifold of dimension $2n.$ Let $e \in H^2(X,\mathbb{Z})$ be the class of half of the exceptional divisor on a deformation of $X$ which is the Kummer variety of dimension $2n$ for some complex two-dimensional  torus. We identify $H^2(X,\mathbb{Z})$ and $H^2(Y,\mathbb{Z})$ for any deformation~$Y$ of $X$ in $\mathrm{Teich},$ since topologically they are the same. Half of the exceptional divisor on a deformation of $X$ which is  Kummer varieties of dimension $2n$ for some complex two-dimensional  torus is then $\gamma(e),$ where $\gamma \in \mathrm{Mon}(X).$    By 
$$
\hat{e}\,=\,\frac{e}{2(n+1)} \in H_2(X,\mathbb{Z})
$$
\noindent we denote its dual class. We have an analogue of Lemma \ref{RHK3}:

\begin{Lemma}\label{RHA}

Let $A$ be a complex two-dimensional  torus. Let $C \subset A$ be a smooth curve of genus $g.$  Assume that there is a linear system $g^1_k$ on $C$ with $2 \leqslant k \leqslant n+1.$  Then there is a map $f \colon C \to \mathrm{Kum}^n(A)$ such that its image is a rational curve $R$ which is numerically equivalent to $C-(g-1+k)\hat{e}$ in~$\mathrm{N}_1(\mathrm{Kum}^n(A),\mathbb{Z}).$

\end{Lemma}
\begin{proof}
The proof is similar to the proof of Lemma \ref{RHK3}. We only need to prove that one can translate $C$ by some point $x \in A$ such that the summation function on every element $x_1+x_2+...+x_{n+1}$ of the image  of $f$ is equal to zero.  There are maps
$$
C \xrightarrow{f} \mathbb{P}^1 \xrightarrow{s} A,
$$

\noindent where $f$ is induced by the linear system $g^1_k$ and $s$ is a summation. The function~$s$ can map $\mathbb{P}^1$ either in $\mathbb{P}^1,$ or in the point. But there is no rational curve in $A.$ So~\mbox{$s:\mathbb{P}^1 \to \text{pt}.$} Translating the curve $C,$ we get $\text{pt}\,=\,0.$

\end{proof}

We have the following important property which will be very useful for proving the minimality of curve in the process of proving Theorem~\ref{TheoremKummer}.

\begin{Lemma}\label{lemma:Kummerk>2g-2}
Let $M=\mathrm{Kum}^n(S)$ for a complex two-dimensional  torus $A$ such that~\mbox{$\mathrm{NS}(A)=\langle x \rangle$} and let~\mbox{$C \subset S$} be a smooth curve of genus $g \geqslant 1$ which corresponds to the class $x \in H^2(A, \mathbb{Z}).$ Let $t=2(n+1).$ Let $e \in  H^{1,1}(M,\mathbb{Z})$ be a class of half of the exceptional divisor on~$M.$ Assume that $\beta=rx-be \in H^{1,1}(M,\mathbb{Z})$ is a primitive class such that $r \in \mathbb{N},$
\begin{equation}\label{eq:Kummerk>2g-2b}
b \in [1,\frac{r}{2}] \quad \text{is an integer number,}
\end{equation}
\noindent and 
\begin{equation}\label{eq:Kummerk>2g-2bq>0}
q(\beta)<0.
\end{equation}
\noindent Then $k=\frac{bt}{r}-g+1>2g-2.$

\end{Lemma}

\begin{proof}
If $g=1,$  we are done since $\frac{br}{t}>0.$ So assume that $g \geqslant 2.$ By \eqref{eq:Kummerk>2g-2bq>0} we obtain
\begin{equation}\label{eq:Kummerfrac<b/r}
\sqrt{\frac{2g-2}{t}} < \frac{b}{r}.
\end{equation}
\noindent So by \eqref{eq:Kummerk>2g-2b} we get
\begin{equation}\label{eq:Kummersqrt>2}
\sqrt{\frac{t}{2g-2}} > 2. 
\end{equation}
\noindent By \eqref{eq:Kummerfrac<b/r} and \eqref{eq:Kummersqrt>2} we have the following inequalities
$$
k=\frac{bt}{r}-g+1>\sqrt{\frac{2g-2}{t}}-g+1=(2g-2)\left(\sqrt{\frac{t}{2g-2}} -\frac{1}{2}  \right)>2g-2.
$$

\end{proof}

\begin{Lemma}\label{lemma:Kummerminimalcurves}
Let $M=\mathrm{Kum}^n(A)$ for a complex two-dimensional  torus $A$ such that~\mbox{$\mathrm{NS}(A)=\langle x \rangle$} and let~\mbox{$C \subset S$} be a smooth curve of genus $g \geqslant 1$ which corresponds to the class $x \in H^2(A, \mathbb{Z}).$ Let $t=2(n+1).$ Let $e \in  H^{1,1}(M,\mathbb{Z})$ be a class of half of the exceptional divisor on~$M.$ Assume that $\beta=rx-be \in H^{1,1}(M,\mathbb{Z})$ is a primitive class such that $r \in \mathbb{N},$ $b \in [1,\frac{r}{2}]$ is an integer number, and $q(\beta)<0.$ Then there is a rational curve $R \subset M$ corresponds to the class $\beta.$ Moreover, the curve $R$ varies in the variety $\mathcal{Z} \subset M$ which is a rational quotient such that dimension of  its fibres  is  equal to codimension of $\mathcal{Z}$ in $M.$ 
\end{Lemma}

\begin{proof}
Let us consider $\widehat{\beta}=x-\frac{b}{r}e \in H_2(X,\mathbb{Q}).$ Let $m=\frac{t}{r}$ We have 
$$
\widehat{\beta} = x-\frac{mb}{t}e \equiv C-mb\hat{e}.
$$

\noindent Let us prove that there is a linear system $g^1_k$ on $C$ such that 
\begin{equation}\label{eq:equationofkA'}
k=mb-g+1=\frac{bt}{r}-g+1.
\end{equation}
\noindent By Lemma~\ref{lemma:Kummerk>2g-2} we get 
\begin{equation}\label{eq:k>2g-2minimalityKummer}
k>2g-2.
\end{equation}

If we prove that  the gonality of $C$ is less than or equal to $k,$ then by Lemma~\ref{RHA} we get that the linear system $g^1_k$ gives a rational curve $R$ on $M.$ To do this we need to use the inequality for the gonality in~\cite[Chapter~2, page 261]{grif}. In other words, we  have to prove that
$$
\floor*{\frac{g+1}{2}}+1 \leqslant  k.
$$
\noindent For $g \geqslant 2$ this immediately follows from Lemma~\ref{lemma:Kummerk>2g-2}. For~\mbox{$g=1$} it is enough to prove that $k \geqslant 2.$ Form \eqref{eq:equationofkA'} it is enough to prove that~\mbox{$\frac{bt}{r} \geqslant 1,$} which is true. This means that there is  a rational curve $R$ which corresponds to the class $\beta.$

Now let us find the MBM locus $\mathcal{Z}$ of the curve $R.$ Assume that the genus of the curve~$C$ is $g>1.$  By \eqref{eq:k>2g-2minimalityKummer} we get that~$\mathrm{Sym}^k(C)$ is a projective bundle over the Jacobian~$J(C)$ of $C.$ Also by Hirzebruch-Riemann-Roch theorem we get that 
$$
\dim H^0(A,\mathcal{O}_A(C))=g-1.
$$
\noindent By~\cite[Proposition \rom{1}.8.1]{Milne} only finitely many translations of $C$ by the elements of~$A$ are linear equivalent to $C.$ Denote this subgroup of $A$ by $H.$ This means that~\mbox{$A/H \times \mathbb{P}^{g-2}$} parametrizes the curves of genus $g$ on the surface $A.$

We have the fibration over $\mathbb{P}^{g-1}$ with the Jacobians as the fibres. In other words, we get the fibration 
$$
j \colon \mathcal{J} \to \mathbb{P}^{g-1},
$$
\noindent  such that the fibre of $j$ over a point $D \in \mathbb{P}^{g-2},$ where $D \subset A$ is a curve of genus $g,$ is equal to the Jacobian of $D.$ Let 
\begin{equation}\label{eq:Kummerminimalsymprojbundle}
\phi \colon \mathcal{P} \to \mathcal{J}
\end{equation}
\noindent  be a fibration such that the fibre over a point $J \in \mathcal{J}$ is $\mathbb{P}^{k-g}$ which gives us the fibration
$$
\mathrm{Sym}^k(D) \xrightarrow{\mathbb{P}^{k-g}} J(D)
$$
\noindent for a curve $D \in \mathbb{P}H^0(A,\mathcal{O}_A(C))^{\vee}$ of genus $g.$ Let $\mathcal{A}$ be a fibre of $\mathcal{J}$ over zero of the summation map to $A$ and let $\mathcal{P}_{\mathcal{A}}$ be a restriction of the fibration  \eqref{eq:Kummerminimalsymprojbundle} on $\mathcal{A}.$ The  completion of~$\mathcal{Z}$ is thus $\mathcal{P}_{\mathcal{A}} \times A/H \times  \mathrm{Sym}^{n+1-k}(A).$  Note that $\mathcal{A}$ is called the Debarre system and it is known that $\mathcal{A}$ is deformation equivalent to $\mathrm{Kum}^{g-2}(\widehat{A})$ (see~\cite[Theorem 3.4]{Debarre}). In other words, the desingularization of $\mathcal{Z}$ is deformation equivalent to $\mathbb{P}^{k-g}$-bundle over the product of IHSM and a two-dimensional complex torus. The dimension of $\mathcal{Z}$ is equal to 
$$
2n+2-2k+2+g+g-2-2+k-g=2n-k+g.
$$
\noindent  The dimension of the fibre of its rational quotient  is equal to $k-g$ which is the codimension of the MBM locus in $M.$

Assume that $g=1.$ Then from the exact sequence 
$$
0 \to \mathcal{O}_A \to \mathcal{O}_A(C) \to \mathcal{O}_C(C) \to 0
$$
\noindent we get that $\dim H^0(A,\mathcal{O}_A(C))$ is either $1,$ or $2.$ If it is equal to $2,$ we get that $A$ is an elliptic fibre over $\mathbb{P}^1,$ which is impossible. So we get that  
$$
\dim H^0(A,\mathcal{O}_A(C))=1.
$$
\noindent This means that the completion of~$\mathcal{Z}$ is a product of a fibre over zero of the summation map $\mathrm{Sym}^k(C) \to C$ and $\mathrm{Sym}^{n+1-k}(A).$ In other words, the completion of~$\mathcal{Z}$ is isomorphic to 
$$
\mathbb{P}^{k-1} \times \mathrm{Sym}^{n+1-k}(A).
$$
\noindent Note that in this case $C$ is a subgroup in $A.$ So one does not need to take into account the translations of $C,$ since these parameters are in the component $\mathrm{Sym}^{n+1-k}(A)$ in this case.     The dimension of $\mathcal{Z}$ is equal  to $2n-k+1$ and the dimension of the fibre of its rational quotient is equal to the codimension of $\mathcal{Z}$ in $M.$ This completes the proof.

\end{proof}

Now we are ready to prove Theorem \ref{TheoremKummer}.

\begin{proof}[Proof of Theorem \ref{TheoremKummer}.]
Let $\alpha$ be a primitive element in $H^2(X,\mathbb{Z})$ such that~\mbox{$q(\alpha)<0.$} The proof of this theorem is similar to the proof of Theorem~\ref{TheoremK3}. If MBM classes~$\alpha$ and~$e$ are in the same monodromy orbit, then there is a deformation of $X$ preserving~$e$, which is isomorphic to $\mathrm{Kum}^n(A)$ with $\mathrm{Pic}(\mathrm{Kum}^n(A))\,=\,\langle e \rangle,$ i.e.~$A$ does not have curves. Note that $e$ is MBM, because it corresponds to a fibre $R$ of the ruling over the diagonal on $\mathrm{Sym}^{n+1}(A),$ intersecting with $\mathrm{Kum}^n(A).$  This ruling is MBM locus of $R.$  Thus, by Theorem \ref{Torelli} there is a deformation of~$X$ preserving~$e,$ which is  isomorphic to $\mathrm{Kum}^n(A)$ with 
$$
\mathrm{Pic}(\mathrm{Kum}^n(A))\,=\,\langle  e \rangle,
$$
\noindent i.e. $A$ does not have curves. We have
$$
q(\widehat{e})=-\frac{1}{2(n+1)}=q(\widehat{\alpha}).
$$

\noindent So assume that $\alpha$ and~$e$ are not in the same monodromy orbit.  We divide the proof into several steps.

\textbf{Step 1.} \emph{In this step we recall that there is a deformation $Y$ of $X$ within~$\textup{Teich}_{\alpha}$ with Picard lattice $\mathrm{Pic}(Y)=\langle \alpha, e \rangle$  such that $Y$ is a Kummer variety.}  This immediately follows from  Corollary~\ref{Corollary:deformation}. So $Y=\mathrm{Kum}^n(A')$ for a complex two-dimensional torus $A'$ with 
$$
\mathrm{NS}(A')=\langle y \rangle=e^{\perp} \; \text{in} \; \langle \alpha,e \rangle.
$$

\textbf{Step 2.} \emph{In this step we deform $Y$ to $Z$ in $\textup{Teich}_{\alpha}$ such that 
$$
\mathrm{Pic}(Z)=\langle \alpha, \gamma^{-1}(e) \rangle
$$
\noindent and $\gamma \in \mathrm{Mon}(X)$  is chosen to normalize $\gamma(\alpha)$ in a certain way.} So let us prove the following lemma.

\begin{Lemma}\label{lemmaAVA}

Let $Y$ be the Kummer variety $\mathrm{Kum}^n(A')$ of a complex two-dimensional torus. Assume that $\mathrm{Pic}(Y)=\langle \alpha, e\rangle,$  where   
$$
\alpha=ay+ce \in H^{1,1}(Y,\mathbb{Z})
$$
\noindent is a primitive class with~\mbox{$q(\alpha)<0$}   for~\mbox{$y \in \mathrm{NS}(A').$}  Then there is~$\gamma \in \textup{Mon}(Y),$ such that $\gamma(\alpha)=rx-be$  with primitive $x \in H^2(A',\mathbb{Z}),$ and with integer numbers $b \in [1,\frac{r}{2}]$ and~\mbox{$r=\gcd(a,t),$}  where~\mbox{$t=2(n+1).$}

\end{Lemma}

\begin{proof}
The proof of this lemma is similar to the proof of Lemma~\ref{lemmaAVK3}. However, in this case we get $b \in [1,\frac{t}{2}],$ because if $b=0,$ then $\alpha$ lies in the same monodromy orbit as an element of $H^2(A',\mathbb{Z})$ for some complex two-dimensional  torus $A'.$ But  smooth curves on $A'$ have non-negative squares. This contradicts with our assumptions that~$\alpha$ is of negative square.

\end{proof}

So consider the point in $\mathrm{Teich}_{\alpha}$ which corresponds to the variety $Z$ such that 
$$
\mathrm{Pic}(Z)=\langle \alpha, \gamma^{-1}(e) \rangle,
$$
\noindent  where $\gamma \in \mathrm{Mon}(Y)$ is as in Lemma~\ref{lemmaAVA}.  The variety $Z$ is bimeromorphic to the Kummer variety $\mathrm{Kum}^n(A)$ with a complex two-dimensional torus $A$ with 
$$
\mathrm{NS}(A)=\langle x \rangle,
$$
\noindent where $x \in \mathrm{NS}(A')$ as in Lemma~\ref{lemmaAVA}. Note that in any monodromy orbit there is only one element which satisfies the conditions of Lemma \ref{lemmaAVA}.

\textbf{Step 3.} \emph{In this step we give necessary condition for $\alpha$ to be MBM.}

\begin{Lemma}\label{lemma:Athenq>0}
Let $Y$ be an IHSM of Kummer type of dimension $2n \geqslant 4.$ Let 
$$
\alpha \in H^{1,1}(Y,\mathbb{Z})
$$
\noindent  be a primitive class such that $\alpha$ is MBM. Assume that $Y$ is deformation equivalent to the variety $Z$ such that $\alpha \in H^{1,1}(Z,\mathbb{Z})$ and there is $\gamma \in \mathrm{Mon}(Y)$ such that $\gamma(Z)$ is bimeromorphic to~\mbox{$\mathrm{Kum}^n(A)$} for some complex two-dimensional  torus $A.$ Assume that $\gamma(\alpha)=rx-be,$ where $\mathrm{NS}(A)=\langle x \rangle,$  $b \in [1,\frac{r}{2}]$ is an integer number and $r$ divides $t=2(n+1).$ Then~\mbox{$q(x) \geqslant 0.$}

\end{Lemma}

\begin{proof}
 If~\mbox{$q(x)<0,$} then the  complex two-dimensional  torus $A$ does not have any curves. Therefore, by Remark~\ref{remark:vanishingofcurve} the class $\alpha$ is not an MBM class, because the only rational curves on the Kummer variety are of class $\hat{e}.$ If $q(x) \geqslant 0,$ we use the uniform representation of monodromy orbit from Step~2, namely, for every~\mbox{$\alpha \in H^2(X,\mathbb{Z})$} there is $\gamma \in \text{Mon}(X)$ such that 
$$
\gamma(\alpha)=rx-be,
$$
\noindent where $r$ divides~$t,$ the integer number $b$ belongs to the segment \mbox{$[1,\frac{r}{2}]$} and~\mbox{$\gcd(r,b)=1.$}  Thus, we can deform~$X$ to $Z$ such that  $\gamma(Z)$ is bimeromorphic to $\mathrm{Kum}^n(S)$ with  
$$
\mathrm{NS}(A)=\langle x \rangle.
$$ 

\end{proof}

\textbf{Step 4.}  \emph{In this step we prove that the necessary condition on $\alpha$ to be MBM given in Step 3 is also sufficient.}

\begin{Lemma}\label{lemma:Aq>0}
Let $Y$ be an IHSM of Kummer type of dimension $2n \geqslant 4.$ Let 
$$
\alpha \in H^{1,1}(Y,\mathbb{Z})
$$
\noindent  be a primitive class.  Assume that~$Y$ is deformation equivalent to the variety $Z$ such that~\mbox{$\alpha \in H^{1,1}(Z,\mathbb{Z})$} and there is $\gamma \in \mathrm{Mon}(Y)$ such that $\gamma(Z)$ is bimeromorphic to~\mbox{$\mathrm{Kum}^n(A)$} for some complex two-dimensional  torus $A.$ Assume that
$$
\gamma(\alpha)=rx-be,
$$
\noindent where $\mathrm{NS}(A)=\langle x \rangle,$  $b \in [1,\frac{r}{2}]$ is an integer number and $r$ divides $t=2(n+1).$ If~$q(x) \geqslant 0,$ then the class of $\widehat{\gamma(\alpha)} \in H_2(Z,\mathbb{Q})$ on $Z$ is represented by a smooth rational curve.

\end{Lemma}

\begin{proof}
Consider the dual class $\widehat{\gamma(\alpha)} \in H_2(Z,\mathbb{Q})$ of $\gamma(\alpha) \in H^2(Z,\mathbb{Z}).$ Note that 
$$
\widehat{\gamma(\alpha)}=\frac{\gamma(\alpha)}{d(\gamma(\alpha))}=x-\frac{b}{r}e=x-\frac{bt}{r}\widehat{e}.
$$ Assume that $\gamma(Z)$ is bimeromorphic to $\mathrm{Kum}^n(A)$ with $\mathrm{NS}(A)=\langle x \rangle$ and $q(x) \geqslant 0$ such that %or~$Pic(A)=\langle x,x' \rangle$ if $x^2=0$ and
\begin{equation}\label{eqA:b<r/2}
\gamma(\alpha)=rx-be, \;\; r \;\; \text{divides} \;\; t, \;\;  b \in [1,\frac{r}{2}] \;\; \text{is integer and} \;\; \gcd(r,b)=1.
\end{equation}

\noindent So by Lemma~\ref{lemma:Kummerminimalcurves} we get the desired, i.e. we obtain a smooth rational curve on~$Z.$

\end{proof}

\begin{Cor}\label{cor:q>2thusmbmKummer}
Let $Y$ be an IHSM of Kummer type of dimension $2n \geqslant 4.$ Let 
$$
\alpha \in H^{1,1}(Y,\mathbb{Z})
$$
\noindent  be a primitive class.  Assume that~$Y$ is deformation equivalent to the variety $Z$ such that $\alpha \in H^{1,1}(Z,\mathbb{Z})$ and there is $\gamma \in \mathrm{Mon}(Y)$ such that $\gamma(Z)$ is bimeromorphic to~\mbox{$\mathrm{Kum}^n(A)$} for some complex two-dimensional  torus $A.$ Assume that 
$$
\gamma(\alpha)=rx-be,
$$
\noindent  where $\mathrm{NS}(A)=\langle x \rangle,$  $b \in [1,\frac{r}{2}]$ is an integer number and $r$ divides $t=2(n+1).$ If~$q(x) \geqslant 0,$ then the class $\alpha$ is MBM.

\end{Cor}

\begin{proof}
By Lemma~\ref{lemma:Aq>0} we get that the class of $\widehat{\gamma(\alpha)} \in H_2(Z,\mathbb{Q})$ on $Z$ is represented by a smooth rational curve. Since it is of negative square by definition of MBM curve it is enough to prove that the curve $R$ which represents $\widehat{\gamma(\alpha)}$ is minimal. The minimality follows from Remark~\ref{remark:criteriaminimality} and Lemma~\ref{lemma:K3minimalcurves}, since by Theorem~\ref{th:analiticallyequiv} it is enough to find MBM locus of $R$ on $\mathrm{Kum}^n(A).$

\end{proof}

\textbf{Step 5.} \emph{In this step we prove that $\alpha \in H^2(X,\mathbb{Z})$ is MBM  if and only if its rational multiple satisfies the assumption of Theorem~\ref{TheoremKummer}.}  Recall that an element~$\alpha$ is MBM if and only if the manifold $X$ can be deformed so that on the deformed variety the class $\widehat{\alpha} \in H_2(X, \mathbb{Q})$ is represented by an extremal rational curve. By Lemma~\ref{lemmaAVA} the element $\alpha$ lies in the same monodromy orbit as the element 
$$
\alpha'=rx-b'e,
$$
\noindent where $r$ divides $t,$  $b' \in [1,r]$ and $\gcd(r,b)=1.$ Let $t=rs,$ where $s \in \mathbb{N}.$ Note that by Lemma~\ref{GritsenkoK3A} we have $q(\alpha)=q(\alpha')$ and~\mbox{$\delta(\alpha)=\pm \delta(\alpha').$} By direct calculation we get that $\delta(\alpha')=-b's.$ Therefore, we obtain
$$
-\delta(\alpha) = -b's \in [1,n+1].
$$
\noindent As 
$$
q(\widehat{\alpha})=q(x-\frac{b'}{r}e)=2a-\frac{b'^2t}{r^2}=2a-\frac{b'^2s^2}{t}
$$
\noindent for $q(x)=2a,$ where $q(x)$ is even because the lattice of the second cohomologies of a complex two-dimension torus is even. By Step $3$ and $4$ we get that $\alpha$ is MBM if and only if $q(x) \geqslant 0.$ This means that
$$
0 \leqslant 2a <\frac{b^2}{2(n+1)}.
$$

\textbf{Step 6.} \emph{In this step we prove that the numerical conditions on MBM class~$\alpha$ defined by~\eqref{eq:bin0,n-1Kummer},~\eqref{eq:q(alpha)Kummer} and~\eqref{eq:-2<2a<bKummer} uniquely determine its monodromy orbit.}   Let us fix integer numbers $a$ and $b$ which satisfy~\eqref{eq:bin0,n-1Kummer},~\eqref{eq:q(alpha)Kummer} and~\eqref{eq:-2<2a<bKummer}.  Let us find all primitive $\alpha \in H^2(X,\mathbb{Z})$ such that
\begin{equation}\label{eq:deltaq(hata)Kummer}
\delta(\alpha)=\pm b \quad \text{and} \quad q(\widehat{\alpha})=2a-\frac{b^2}{2(n+1)}.
\end{equation}

 Consider $\alpha=fx+de$ on some deformation $Y$ of $X,$ such that $Y$ is bimeromorphic to $\mathrm{Kum}^n(A)$ for some complex torus  $A,$ where $f,d \in \mathbb{Z},$ $e \in H^2(X, \mathbb{Z})$ is the class of half of the exceptional divisor on $Y$ and $x \in H^2(A,\mathbb{Z}).$   Recall (by Definition~\ref{definition:divisib}) that 
$$
d(\alpha)=\gcd(f,2(n+1)).
$$
\noindent Let $f=f' d(\alpha).$  Since $\delta(\alpha)=\pm b,$ by the definition of $\delta(\alpha)$ we get that 
$$
d=\pm \frac{bd(\alpha)}{2(n+1)} + cd(\alpha)
$$
\noindent for some integer $c.$ Note that by our assumptions $d$ is also integer. We have
\begin{equation}\label{eq:alpa=b+cdKummer}
\alpha=f'd(\alpha)x+\left(\pm \frac{bd(\alpha)}{2(n+1)} + cd(\alpha)\right)e.
\end{equation}
\noindent  So by \eqref{eq:deltaq(hata)Kummer}  a primitive $\alpha \in H^2(X,\mathbb{Z})$ satisfies \eqref{eq:deltaq(hata)Kummer} if and only if it is of the form~\eqref{eq:alpa=b+cdKummer} such that 
$$
f'^2q(x)-2(n+1)c^2 \mp 2bc=2a.
$$ 
\noindent By direct computations it is not hard to see that any $\alpha$ of the form \eqref{eq:alpa=b+cdK3} lies in the same monodromy orbit as
$$
\alpha'=d(\alpha)y-\frac{bd(\alpha)}{2(n+1)}e,
$$
\noindent where $y \in H^2(X,\mathbb{Z})$ such that $q(y,e)=0$ and $q(y)=2a.$  Note that $\alpha'$ satisfies~\eqref{eq:deltaq(hata)Kummer}. It is enough to prove that all elements of the form 
$$
\alpha'=mz-\frac{bm}{2(n+1)}e,
$$
\noindent where  
$$
\frac{bm}{2(n+1)} \in \mathbb{N}, \quad \gcd\left(m,\frac{bm}{2(n+1)}\right)=1  \quad \text{and} \quad m \;\; \text{divides} \;\; 2(n+1),
$$
\noindent  which satisfy~\eqref{eq:deltaq(hata)Kummer}, lie in the same monodromy orbit. Indeed, let
$$
r=\gcd(b,2(n+1)).
$$
\noindent Let $b=b' \cdot r$ and $2(n+1)=v \cdot r.$ Then since 
$$
\frac{bm}{2(n+1)}=\frac{b' \cdot m}{v}
$$
\noindent is integer number, we get that $v$ divides $m.$ By   the condition $ \gcd(m,\frac{bm}{2(n+1)})=1$ we get that 
$$
m=v=\frac{2(n+1)}{r}.
$$

\noindent This means that all primitive $\alpha$ which satisfy \eqref{eq:deltaq(hata)Kummer} lie in the same monodromy orbit as the class 
$$
\alpha'=\frac{2(n+1)}{\gcd(b,2(n+1))}z-\frac{b}{\gcd(b,2(n+1))}e,
$$
\noindent where $q(z)=2a$  on some deformation $Y$ of $X,$ such that $Y$ is bimeromorphic to~$\mathrm{Kum}^n(A)$ for some complex torus  $A,$ where  $e \in H^2(X, \mathbb{Z})$ is the class of half of the exceptional divisor on $Y$ and $z \in H^2(A,\mathbb{Z}).$

\end{proof}

\begin{proof}[Proof of Corollary~\ref{cor:TheoremKummer}]
Assertion \ref{1'} follows from  Lemma~\ref{GritsenkoK3A}. Assertion \ref{2'} follows from Lemma~\ref{lemma:Athenq>0} and Corollary~\ref{cor:q>2thusmbmKummer}. Note that if $(a,b)=(0,1),$ then by   Lemma~\ref{GritsenkoK3A} the elements $e$ and $\alpha$ lie in the same monodromy orbits.

\end{proof}

\begin{Cor}
Let $X$ be an IHSM of Kummer type of dimension $2n \geqslant 4.$ Let
$$
\alpha \in H^2(X,\mathbb{Z})
$$
\noindent be an MBM class. Assume that $\alpha$ up to a rational multiple can be represented by
$$
2(n+1) \cdot x-be
$$
\noindent with integer number $b \in [1,n+1]$ on some deformation $Y=\mathrm{Kum}^n(A)$ of $X$ with Picard group $\mathrm{Pic}(Y)=\mathrm{NS}(A) \oplus \mathbb{Z}e,$  such that $A$ is a complex two-dimensional torus,~$\mathrm{NS}(A)=\mathbb{Z}x,$  $x^2 \geqslant 0$ and $e$ is a class of half of the exceptional divisor on $Y.$ Then
$$
g=\frac{q(x)}{2}+1 \leqslant  \lceil \frac{n+5}{4} \rceil-1.
$$
\end{Cor}

\begin{proof}
Since $\alpha$ is MBM we get that $q(\alpha)<0.$ So we have
$$
q(2(n+1) \cdot x-be)=4(n+1)^2x^2-2(n+1)b^2<0.
$$
\noindent Therefore, we get $g<\frac{b^2}{4(n+1)}+1 \leqslant \frac{n+5}{4}.$ Thus, we obtain the inequality 
$$
g \leqslant  \lceil \frac{n+5}{4} \rceil - 1 .
$$ 

%as $g$ is less or equal than the maximal natural number strictly less than $\frac{n+5}{4}.$

\end{proof}

\begin{Cor}[\textbf{of the proof}]\label{iffA}
Let $X$ be an IHSM of Kummer type of dimension~\mbox{$2n \geqslant 4.$} Let $\alpha \in H^2(X,\mathbb{Z})$ and $\widehat{\alpha}\,=\,x-\frac{b}{2(n+1)}e$ be its dual with~\mbox{$b \in [1,n+1]$} on some deformation $Y$ of $X,$ such that $Y \simeq \mathrm{Kum}^n(A),$ where~$A$ is a complex two-dimensional  torus such that $\mathrm{NS}(A)=\mathbb{Z}x.$ Then $\alpha$ is MBM if and only if $q(x) \geqslant 0$ and~\mbox{$q(\alpha)\,<\,0.$}
\end{Cor}

\begin{Cor}
Let $X$ be an IHSM of Kummer type of dimension~\mbox{$2n \geqslant 4.$}   Let 
$$
\alpha \in H^2(X,\mathbb{Z})
$$
\noindent and $\widehat{\alpha} \in H^2(X, \mathbb{Q})$ be its dual class. Let $\alpha$ be MBM. Then we get
$$
q(\widehat{\alpha}) \geqslant -\frac{n+1}{2}.
$$
\noindent Moreover, the equality holds if and only if $\alpha$ is a class of the line in the Lagrangian space~$\mathbb{P}^n$, which is the fibre of the summation map $\mathrm{Sym}^{n+1}(C) \to C$ over zero for an elliptic curve $C \subset A.$  In this case  the MBM locus of $\alpha$ on the variety $X$  is bimeromorphic to  the fibre of the summation map $\mathrm{Sym}^{n+1}(C) \to C$ over zero.

\end{Cor}

\begin{proof}

By Theorem~\ref{TheoremKummer} we get that either $q(\widehat{\alpha})=-\frac{1}{2(n+1)},$ or $q(\widehat{\alpha})\,=\,2a-\frac{b^2}{2(n+1)},$  where $b \in [1,n+1]$ is an integer number and  $a$ is an integer number which satisfies the inequalities
$$
 0 \leqslant 2a<\frac{b^2}{2(n+1)}.
$$

\noindent We get
$$
q(\widehat{\alpha})=2a-\frac{b^2}{2(n+1)} \geqslant -\frac{n+1}{2}.
$$

\noindent In the inequality the equality holds if and only if  $q(x)=0$ and $b=n+1.$ Therefore, by~\cite[Example 6.3]{Tschinkel} we get that~$\alpha$ is a class of the line in the Lagrangian subspace. By Lemma~\ref{RHA} the extremal curves on $\mathrm{Kum}^n(A)$ which correspond to the class~$\alpha$  are the fibre of the summation map $\mathrm{Sym}^{n+1}(C) \to C$ over zero for the elliptic curve~\mbox{$C \subset A.$} Indeed, by Lemma~\ref{RHA} we have 
$$
\widehat{\alpha}=C-(n+1)\widehat{e},
$$
\noindent where $C$ is an elliptic curve on a two-dimensional complex torus.  Let $R$ be a rational curve in the class $\widehat{\alpha}.$ This means that  that $C$ maps to the rational curve $R$ by the map of degree $n+1$ and the intersection of $R$ and the exceptional divisor on the Kummer variety is equal to $n+1.$ All linear systems on $C$ of degree $n+1$ which give the covering of a rational curve on a Kummer variety are parametrized by the fibre of 
$$
\mathrm{Sym}^{n+1}(C) \to C
$$
\noindent over zero. The fibre is isomorphic to $\mathbb{P}^n.$ Therefore, by  Theorem~\ref{th:analiticallyequiv} in this extremal case  for the Kummer type variety $X$ the MBM locus is  bimeromorphic to  the fibre of the summation map $\mathrm{Sym}^{n+1}(C) \to C$ over zero.

\end{proof}

Now we prove Proposition  \ref{BMeqAVabel}.

\begin{proof}[Proof of Proposition \ref{BMeqAVabel}]
According to Yoshioka theory we get that the walls are the hyperplanes in $\mathrm{Pos}(X)$ orthogonal to $a$ with $a^2 \geqslant 0$ (the square of $a$ in the sense of the Mukai scalar product). The vector $a\,=\,(u,\varkappa,s)$ belongs to 
$$
H^0(A,\mathbb{Z}) \oplus \mathrm{NS}(A) \oplus H^4(A,\mathbb{Z}),
$$
\noindent where~$A$ is a complex two-dimensional  torus. Turning to our situation we need to project this vector on $v^{\perp},$ i.e. consider $z\,=\,v^2a-(v,a)v.$ Obviously, this vector is orthogonal to $v.$ By simple calculations we have 
$$
z\,=\,v^2\varkappa-(a,e)e,
$$
\noindent where without loss of generality we put $v\,=\,(1,0,-1-n)$ and $e\,=\,(1,0,n+1).$ So the walls in $\mathrm{Pos}(X)$ are hyperplanes orthogonal to $z\,=\,v^2\varkappa-(a,e)e.$

By Theorem \ref{TheoremKummer} and its proof $z$ is MBM if and only if we can deform $X$ to $Z$ such that $\gamma(Z)$ is bimeromorphic to  $\mathrm{Kum}^n(A)$ for some complex two-dimensional torus~$A.$ Moreover, the following properties hold: the  Picard group is 
$$
\mathrm{Pic}(Z)\comm{_{\mathbb{Q}}}=\langle z, \gamma^{-1}(e) \rangle
$$
\noindent with $\gamma(z)=tx-be,$ where $b \in [1,\frac{t}{2}]$ is an integer, $q(x) \geqslant 0$ such that 
\begin{equation}\label{eqA:BMq=}
q(z)=q(\gamma(z))
\end{equation}
\noindent  and $\gamma \in \text{Mon}(Z).$ As $(a,e)=-s-\frac{tu}{2},$ we have $b=\frac{tu}{2}-s,$ because by conditions of Theorem \ref{Theorem12.1Kummer} we have 
$$
0 <  \frac{tu}{2}-s=-s-\frac{tu}{2}+tu \leqslant  \frac{t}{2}.
$$
\noindent From \eqref{eqA:BMq=} we get
$$
t^2q(\varkappa)-\left( \frac{tu^2}{4}+s^2+tus \right) t\,=\,t^2q(x)-\left( \frac{tu^2}{4}+s^2-tus \right) t.
$$

\noindent So we obtain $q(x)=q(\varkappa)-2us=a^2 \geqslant 0.$ Thus, $z$ is an MBM class.

In the opposite direction, assume that $\alpha$ is an MBM class. Thus, there is a deformation $Z$ of $X$ such that  $\gamma(Z)$ is bimeromorphic to $\mathrm{Kum}^n(A)$ for some complex two-dimensional  torus $A,$ where 
$$
s\gamma(\alpha)=tx-be
$$
\noindent  for some positive integer $s$ which divides $t.$ We have $b \in [1,\frac{t}{2}].$ Thus, let us take 
$$
a=\left(u,\varkappa,\frac{tu}{2}-b\right)
$$
\noindent  for some integer number $u$ and $q(\varkappa)=q(x)+2u\left(\frac{tu}{2}-b\right).$ So we get that $a$ satisfies the conditions of Theorem \ref{Theorem12.1Kummer}. Indeed, 
$$
a^2=q(\varkappa)-2u\left(\frac{tu}{2}-b\right)=q(x) \geqslant 0
$$
\noindent and
$$
(a,v)=\frac{tu}{2}-\frac{tu}{2}+b=b
$$
\noindent Thus, we get $0 <  (a,v) \leqslant  \frac{t}{2}.$ Moreover, we have that 
$$
z=t\varkappa-(a,e)e=t\varkappa-be
$$
\noindent is a projection of $a$ on $v^{\perp}.$

\end{proof}

%
%--------------%
% BIBLIOGRAPHY %
%--------------%
%
%\bibliography{bib}
%\bibliographystyle{amsplain}
%

\providecommand{\bysame}{\leavevmode\hbox to3em{\hrulefill}\thinspace}
\providecommand{\MR}{\relax\ifhmode\unskip\space\fi MR }
\providecommand{\MRhref}[2]{%
  \href{http://www.ams.org/mathscinet-getitem?mr=#1}{#2}
}
\providecommand{\href}[2]{#2}

\end{document}